\documentclass[11pt,final]{amsart}
\usepackage{amsmath,amssymb,amsthm,amsfonts,mathrsfs,amsopn}
\usepackage[all]{xy}
\usepackage{hyperref}
\usepackage{color}
\usepackage{esint}
\usepackage{mathtools}
\mathtoolsset{showonlyrefs}
\allowdisplaybreaks[1]
\usepackage{slashed}

\usepackage{tikz-cd}

\usepackage[margin=0.9in]{geometry}

\usepackage[obeyDraft]{todonotes}

\newcommand{\Abracket}[1]{\left<#1\right>} 
\newcommand{\parenthesis}[1]{\left(#1\right)} 
\newcommand{\braces}[1]{\left\{#1\right\}} 

\newcommand{\R}{\mathbb{R}}

\newcommand{\bZ}{\mathbb{Z}}

\newcommand{\dd}{\mathop{}\!\mathrm{d}}
\newcommand{\eps}{\varepsilon}
\newcommand{\p}{\partial}

\newcommand{\D}{\slashed{D}}

\newcommand{\snabla}{\slashed{\nabla}} 

\newcommand{\II}{\mathrm{I\hspace*{-0.4ex}I}} 

\newcommand{\CHI}{\mathbf{B}}

\DeclareMathOperator{\dv}{\dd{v}}

\DeclareMathOperator{\Eigen}{Eigen}
\DeclareMathOperator{\End}{End}

\DeclareMathOperator{\id}{Id}
\DeclareMathOperator{\Ker}{Ker}

\DeclareMathOperator{\Span}{Span}

\DeclareMathOperator{\tr}{tr}
\DeclareMathOperator{\vol}{vol}

\newtheorem{thm}{Theorem}[section]

\newtheorem{lemma}[thm]{Lemma}
\newtheorem{prop}[thm]{Proposition}
\newtheorem{cor}[thm]{Corollary}

\newtheorem{rmk}[thm]{Remark}
\newtheorem{eg}{Example}

\usepackage[
backend=biber,
style=numeric,
sorting=nyt
]{biblatex}

\title[Weighted eigenvalues of Dirac operators]{Weighted eigenvalues of Dirac operators: \\  complete continuity and comparison}


\author[Z. Qiu]{Zixuan Qiu}
\address{Zixuan Qiu, School of mathematics and statistics, Beijing Institute of Technology, Zhongguancun South Street No. 5, 100081 Beijing, P.R.China.}
\email{ZixuanQiu@bit.edu.cn}

\author[R. Wu]{Ruijun Wu}
\address{Ruijun Wu, School of mathematics and statistics, Beijing Institute of Technology, Zhongguancun South Street No. 5, 100081 Beijing, P.R.China.}
\email{ruijun.wu@bit.edu.cn}

\begin{document}

\begin{abstract}
    We give a min-max characterization of all of the weighted Dirac eigenvalues, and show that the weighted eigenvalues and eigenspaces of Dirac operators are continuous with respect to weak~$L^p$ convergence of the weight, for any~$p>n$.
    Moreover, we establish a comparison result for such weighted eigenvalue problems including the case of nontrivial harmonic spinors.
\end{abstract}

\maketitle
\enlargethispage{2pt}


\section{Introduction}\label{sect:Intro}

    We consider a weighted eigenvalue problem for Dirac operators of the form
    \begin{align}\label{eq:weighted eigenvalue problem}
        \D_g\psi= \lambda A \psi,\qquad \mbox{ in } M,
    \end{align}
    where~$(M,g)$ is a connected spin Riemannian manifold, with or without boundary,~$\D_g$ is the Dirac operator on the given spinor bundle~$\Sigma_g M$ and~$\psi\in \Gamma(\Sigma_g M)$ is a spinor field,~$A\in\End(\Sigma_g M)$ is a symmetric endomorphism of the spinor bundle which is positive definite almost everywhere.
    It is well known that~\eqref{eq:weighted eigenvalue problem} has a sequence of real eigenvalues which diverge to both~$+\infty$ and~$-\infty$, and the eigenspinors constitute a complete basis of the weighted~$L^2$ space of spinors, see e.g.~\cite{Ginoux2009Dirac,LawsonMichelsohn1989spin}.

    Equations of the form~\eqref{eq:weighted eigenvalue problem} arise in various geometric and variational problems.
    For example, in the local spinorial Weierstrass representation of surfaces in~$\R^3$~\cite{Friedrich1998Representation,Taimanov1997modified, Taimanov1998Weierstrass}, the associated spinor~$\psi$ satisfies an equation of the form
    \begin{align}
        \D_g\psi= H\psi
    \end{align}
    and has constant length, where~$H$ is a scalar function which turns out to be the mean curvature of the embedded surface.
    Also in super Liouville equations~\cite{Han2025superLiouville,  Jevnikar2020existence,Jevnikar2021SuperLiouville,Jevnikar2021sinhGordon, Jost2007super}, which are our main motivation, the spinor has to satisfy
    \begin{align}
        \D_g\psi= \lambda e^u\psi
    \end{align}
    for some~$u\in H^1(M,g)$ and some~$\lambda>0$.
    In the nonlinear spinorial Yamabe problem~\cite{Bartsch2022Yamabe,Isobe2024Yamabe,Isobe2023Yamabe}, we consider the associated weighted eigenvalue problem, which takes the form
    \begin{align}
        \D_g\psi=\lambda |\varphi|^{\frac{2}{n-1}}\psi
    \end{align}
    where~$\varphi$ is a nonzero solution to the spinorial Yamabe equation.
    There are, of course, other nonlinear weighted eigenvalue problems for Dirac type operators which take more complicated forms.
    In the present work we focus on the weighted linear eigenvalue problem of the general form~\eqref{eq:weighted eigenvalue problem}, though the special case where~$A$ reduces to a positive function, i.e.~$A=\rho\in C^\infty(M,\R_+)$, is more common in various contexts.

    When~$A=\id$, this reduces to the classical eigenvalue problem of Dirac operators, which has been a source of new mathematics in the last fifty years, including the Friedrich's eigenvalue estimates, the heat flow approach to index theory, the spectral flow and eta invariants, and so on.
    We cannot list all the literature here but refer to e.g.~\cite{Ginoux2009Dirac} and the references therein.

    The continuous dependence on various parameters is also the theme of a lot of works, among them we mention the continuity of the Dirac spectrum in the Riemannian metrics shown by Nowaczyk~\cite{Nowaczyk2013continuity}.
    Note that the spectral map has to be chosen carefully and the space of all spectra is endowed with an arsinh metric.
    In differential geometry, it is usually dealing with smooth deformations and hence the eigenvalues, sometimes even the eigenspinors, are smooth with the deformation parameter.
    This is in particular the case in the study of spectral flows~\cite{Robbin1995spectral}.
    It is intriguing to ask, as is usually the situation in analysis, what happens if the deformation is not smooth but merely continuous in some weak sense.

    Concerning the weighted eigenvalue problem~\eqref{eq:weighted eigenvalue problem}, it is natural to ask~\emph{ in what sense of continuous deformation of the weights~$A$ will the eigenvalues, or even the eigenspaces, remain continuous. }
    In the present work we try to give a weak sufficient condition.

    We remark that the continuity of eigenvalues and eigenfunctions on parameters is a classical subject which is fundamental in analysis and geometry.
    In~\cite{CourantHilbert1953} the one-dimensional case is carefully analyzed.
    Lou and Yanagida~\cite{Lou2006minimization} make important progress where the weights may fail to be definite but they still establish continuity and comparison properties of the weighted eigenvalues.
    The weighted eigenvalue problems for Laplacian operators play an important role in variational and geometric problems, see~\cite{AmbrosettiProdi1995primer, LiYau1983eigenvalue}
    Weighted eigenvalue problems of other operators are also considered in e.g.~\cite{Cuesta2009weighted}.

    To state our result we take the following hypothesis:
    \begin{align}\label{hypo} \tag{H}
        & A_m, \, A \mbox{ are symmetric, positive definite almost everywhere},\\
        &\mbox{ and } A_m\rightharpoonup A \mbox{ weakly in } L^p, \mbox{ for some } 2 \le n<p<\infty.
    \end{align}
    The weak~$L^p$ convergence of~$A_m$ to~$A$ means that for any spinors~$\phi\in L^q,\eta\in L^r$ with~$\frac{1}{q}+\frac{1}{r}=1-\frac{1}{p}$, we have
    \begin{align}
        \int_M \Abracket{A_m\,\phi, \eta}\dv_g \to \int_{M} \Abracket{A\,\phi,\eta}\dv_g, \qquad \mbox{ as } m\to+\infty.
    \end{align}
    Here weak convergence is understood in the finite-rank endomorphism bundle, with the dual pairing against $L^{p/(p-1)}$ endomorphisms. We use the norm
    \[
        \|A\|_{L^p}=\left(\int_M|A(x)|_{op}^p\dv_g\right)^{1/p}.
    \]
    In particular, $\sup_m\|A_m\|_{L^p}<\infty$ and $\|A\|_{L^p}\le\liminf_m\|A_m\|_{L^p}$. Weak convergence of the endomorphisms does not imply weak convergence of their pointwise norms.
    In particular, when~$A_m=H_m \id_{\Sigma_g M}$, this reduces to the weak~$L^p$ convergence of~$L^p$ functions~$H_m$.

    We will write~$\Eigen(L;\lambda)$ for the eigenspace of an operator~$L$ associated to eigenvalue~$\lambda$.
    Moreover, we will see in Section~\ref{sect:minmax} and Section~\ref{sect:boundary} that the weighted eigenvalues for Dirac operators, both in closed and boundary case, can be listed as
    \begin{align}
        -\infty\leftarrow \cdots \leq \lambda_{-2}(A)\leq \lambda_{-1}(A) < 0< \lambda_1(A) \leq \lambda_2(A)\le \cdots \to +\infty.
    \end{align}
    where~$\lambda_j(A)$ are the nonzero eigenvalues counted with multiplicities.
    We denote the~$\ell$-th distinct eigenvalues by~$\mu_\ell(A)$, which are also listed in a strictly increasing order, and let~$H_\ell(A)=\Eigen(A^{-1}\D_g;\mu_\ell)$ be the associated eigenspace with~$h_\ell=\dim H_\ell(A)$, and~$P_\ell(A)$ the associated projector.
            That is, for~$\ell>0$, let~$\sigma_{\ell} = \sum_{p = 1}^{\ell} h_p$, then
            \begin{align}
                H_1(A)=&\Span\braces{\varphi_j(A)\mid 0<j\leq h_1}, \\
                H_\ell(A)=&\Span\braces{\varphi_j(A)\mid \sigma_{\ell-1} < j \leq \sigma_\ell },\quad  \ell>1,
            \end{align}
            for~$\ell<0$, the eigenspaces and projectors are similarly defined.
            Moreover, for each~$A_m$, instead of collecting the eigenspaces according to the eigenvalues of~$A_m^{-1}\D_g$, we collect them according to the eigenvalues of~$A^{-1}\D_g$, namely
            \begin{align}
                H_1(A_m)=&\Span\braces{\varphi_j(A_m)\mid 0<j\leq h_1}, \\
                H_\ell(A_m)=&\Span\braces{\varphi_j(A_m)\mid \sigma_{\ell-1} < j \leq \sigma_\ell   },\quad \ell>1,
            \end{align}
            and~$H_\ell(A_m)$ for~$\ell<0$ is similarly defined.
            The projector onto~$H_\ell(A_m)$ is denoted by~$P_\ell(A_m)$ for each~$\ell\neq 0$.

    All eigenspinors are normalized in the weighted $L^2$ inner product. The projectors are the weighted spectral projectors; for example,
    \[
       P_\ell(A_m)u=\sum_{\sigma_{\ell-1}<j\le\sigma_\ell}
       \left(\int_M\Abracket{A_m u,\varphi_j(A_m)}\dv_g\right)\varphi_j(A_m),\qquad \ell>0.
    \]
    This formula also defines them on the ordinary Sobolev spaces used below. An approximating cluster can contain several distinct eigenvalues.

    \begin{thm}\label{thm:continuity-closed}
        Let~$(M^n,g)$ be a closed Riemannian spin manifold,~$(\Sigma_g M, g^s,\snabla,\gamma)$ be a spinor bundle over~$M$ and~$A_m,A$ be weights satisfying~\eqref{hypo}.
        Then, for each~$k\in \mathbb{Z}\setminus\braces{0}$,
        \begin{enumerate}
            \item[(i)] $\lambda_k(A_m)\to \lambda_k(A)$ as~$m\to+\infty$;
            \item[(ii)] ~$\varphi_k(A_m)$ converges to~$\Eigen(A^{-1}\D_g;\lambda_k(A))$, more precisely,
            \begin{align}
                \lim_{m\to+\infty} \inf_{\phi\in\Eigen(A^{-1}\D_g;\lambda_k(A))} \|\varphi_k(A_m)-\phi\|_{C^\alpha}=0,
            \end{align}
            where~$\alpha \in (0,1)$ is chosen as in Proposition~\ref{prop:regularity}, and, for every $0<s<1$,
            \begin{align}
                \lim_{m\to+\infty} \inf_{\phi\in\Eigen(A^{-1}\D_g;\lambda_k(A))} \|\varphi_k(A_m)-\phi\|_{H^{s}}=0.
            \end{align}
            \item[(iii)] In operator norm on $H^s$, for every $0<s<1$, we have
            \begin{align}
                \lim_{m\to+\infty} P_\ell(A_m)= P_\ell(A), \quad \forall\; \ell\neq 0.
            \end{align}

            \end{enumerate}
    \end{thm}
    Some comments are in order.
    Firstly, if the weights~$A_m$ converge to~$A$ in~$C^1$, then it is well known that the corresponding weighted eigenvalues converge, as well as the associated projectors to the corresponding eigenspaces, see e.g.~\cite{Kato1995perturbation}.
    The point here is that when the weights converge only weakly, we still have the continuity.
    Such continuity is referred to as \emph{complete continuity} in functional analysis~\cite{Taylor1986Intro, WenYangZhang2018complete}.
    Thus what we obtain is the complete continuity of the eigenvalue, the eigenspinors as well as the eigenprojectors with respect to the weights.
    Secondly, the stronger first-order Sobolev convergence is discussed after the proof in Section~\ref{sect:closed}.
    Thirdly, a similar statement can be made for the weighted eigenvalues for Laplacian operators on manifolds, using Remark~\ref{rmk:Laplacian}.

    In general the result fails for~$p=n$; an explicit counterexample is given in Section~\ref{sect:closed}.

    \

    In case~$M$ comes with a smooth boundary~$\p M$, we need to impose suitable boundary conditions.
    Different from the situation of spectral flows in odd dimensions where the Atiyah-Patodi-Singer boundary conditions are usually used, here we impose a local boundary condition, namely the chiral boundary condition
    \begin{align}
        \CHI^+\psi=0, \qquad \mbox{ on } \; \p M.
    \end{align}
    Here~$\CHI^+$ stands for the chiral operator defined along~$\p M$, see Section~\ref{sect:boundary}.
    In this case a similar statement holds:
    \begin{thm}
        The weighted \emph{chiral} eigenvalues and eigenspaces of~\eqref{eq:weighted eigenvalue problem} are continuous with respect to continuous deformations of the weight~$A$ in the weak~$L^p$ topology in the sense of~\eqref{hypo}, for any~$p>n$.
    \end{thm}

    The more precise statement is given in Theorem~\ref{thm:continuity-bdd}.

    The outline of the proof is similar to that in~\cite{WenYangZhang2018complete} where the complete continuity of Laplacian eigenvalue problems in a bounded Euclidean domain with Dirichlet boundary conditions was considered.
    While the Laplacian operator with Dirichlet boundary condition is a positive operator and enjoys nice min-max characterizations, the Dirac operator is strongly indefinite and requires special care.
    The spectrum of this weighted eigenvalue problem consists of infinitely many negative eigenvalues, making it hard to start with the min-max procedure as usual.
    We generalize Ammann's characterization of the first positive eigenvalues, and obtain a new characterization of all nonzero weighted eigenvalues in Proposition~\ref{prop:another minmax} for the weighted closed eigenvalues and Proposition~\ref{prop:another minmax bdd} for the weighted chiral eigenvalues, for which the theory is more involved.
    We will first obtain an upper bound for each positive eigenvalue by the min-max characterization. Uniform estimates for the normalized eigenspinors then give compactness, exclude a zero limit, and identify each limiting eigenvalue.
    By establishing uniform a priori estimates for the weighted eigenspinors, we can finally show the convergence of the eigenspinors to some limit contained in the corresponding eigenspaces.
    Conversely, any~$A$-weighted eigenspinor is shown to be approximated by some~$A_m$-weighted eigenspinors.
    This will be shown by a spectral gap argument.
    Thus the weighted spectral projectors converge as operators, proving the continuity of eigenspaces.

    As a byproduct of the above min-max characterization of weighted eigenvalues, we prove a comparison result for the eigenvalues of~\eqref{eq:weighted eigenvalue problem}, which is also motivated by the analysis of weighted eigenvalue of Laplacian operator with Dirichlet boundary conditions in Euclidean domains.
    \begin{thm}\label{thm:comparison}
        Let~$A_1,A_2\in \End(\Sigma_g M)$ be~$C^1$, symmetric, and positive definite a.e.

        If~$A_1\geq A_2$, then for each~$k>0$ {\rm [}resp. $k<0${\rm ]},
        \begin{align}
            \lambda_k(A_1)\leq \lambda_k(A_2), & & [\mbox{resp.} \qquad \lambda_k(A_1)\geq \lambda_k(A_2). \;  ]
        \end{align}
    \end{thm}

    The rigidity statement about eigenvalue comparison fails for general weights, see the discussion in the end of Section~\ref{sect:minmax}.
    But in the case of scalar weights, we can get some rigidity result, see Corollary~\ref{cor:rigidity-scalar}. 

    If stronger continuity are imposed on the weights, then we can also get better continuity properties of the eigenvalues and eigenspaces, see~\cite{QiuWu2026continuity}. 
    In some case we can also allow the domain to vary. 
    We leave it to a future work.

    The arguments as well as the results can be extended to more general settings.
    We will make a discussion on this point in the end.

    \

    The paper is organized as follows.
    We first give a min-max characterization of the nonzero weighted eigenvalues in Section~\ref{sect:minmax} and a priori estimates of the weighted eigenspinors in Section~\ref{sect:estimates}.
    Then we move to prove the weak continuity of weighted eigenvalues and eigenspaces, for the closed manifolds in Section~\ref{sect:closed} and for manifolds with boundary in Section~\ref{sect:boundary}.
    Finally we discuss some applications of the continuity results in Section~\ref{sect:applications}.
    In the concluding discussion we consider possible extensions to general elliptic operators.

    \

    \noindent{\bf Acknowledgement.} Z.Q thanks Prof. Tongzhu Li for constant support.
    R.W. would like to thank Yuan Lou, Andrea Malchiodi and Zuoqin Wang for helpful conversation on these problems.


\section{Weighted eigenvalue problems for Dirac operators}\label{sect:minmax}

    We will first be concerned with the closed eigenvalue problems.
    The boundary eigenvalue problems will come later in a parallel way.

    \

    Let~$(M^n,g)$ be a closed spin Riemannian manifold, with a given spinor bundle~$\Sigma_g M$.
    We will think of~$\Sigma_g M$ as a real vector bundle of rank~$2^{[\frac n2]+1}$, and temporarily forget about its hermitian structure.
    There exists a Riemannian structure, i.e. a fiberwise inner product~$g^s$, a spin connection~$\snabla$, and a Clifford multiplication~$\gamma$ satisfying the Clifford relation
    \begin{align}
        \gamma(X)\gamma(Y)+\gamma(Y)\gamma(X)= -2g(X,Y)\id_{\Sigma_g M}, \qquad \forall X,Y\in\Gamma(TM).
    \end{align}
    Moreover, they are compatible with each other and hence~$(\Sigma_g M, g^s, \snabla, \gamma)$ is a Dirac bundle in the sense of~\cite[Definition 5.1]{LawsonMichelsohn1989spin}.
    Thus we can define the Dirac operator as the composition of the following arrows
    \begin{align}
        \Gamma(\Sigma_g M) \xrightarrow[]{\snabla}\Gamma(T^*M\otimes \Sigma_g M)
        \xrightarrow{\cong} \Gamma(TM\otimes \Sigma_g M)
        \xrightarrow{\gamma} \Gamma(\Sigma_g M).
    \end{align}
    In terms of a local orthonormal frame~$(e_i)$, we have
    \begin{align}
        \D\psi =\sum_{i} \gamma(e_i)\snabla_{e_i} \psi, \qquad \forall \; \psi\in \Gamma(\Sigma_g M).
    \end{align}
    The Dirac operator is a first order elliptic operator, and is essentially self-adjoint in the space of~$L^2$ spinors provided~$M$ is closed, see e.g.~\cite{Friedrich2000Dirac, Ginoux2009Dirac}.

    \

    Let~$A\in L^p(\End_\R(\Sigma_g M))$, $p>n$, be a fiberwise \textbf{symmetric, positive definite} endomorphism of the spinor bundle, with fiberwise inverse~$A^{-1}$.
    Consider the eigenvalue problems
    \begin{align}
        A^{-1}\D_g\varphi_k =\lambda_k \varphi_k.
    \end{align}
    For smooth positive weights, $A^{-1}\D_g$ is elliptic and self-adjoint for the weighted inner product $\int_M\!\Abracket{A\psi,\phi}\dv_g$.
    For the $L^p$ weights considered here, its domain is
    \[
    \left\{\psi\in W^{1,\frac{2p}{p+1}}(\Sigma_g M):
       \int_M\Abracket{A^{-1}\D_g\psi,\D_g\psi}\dv_g<\infty\right\}.
    \]
    To describe this realization, let $P_0(A)$ denote the weighted orthogonal projection onto $\ker\D_g$. If $f$ belongs to the weighted $L^2$ space and $f\perp_A\ker\D_g$, then
    \[
    \|Af\|_{L^{2p/(p+1)}}\le\|A\|_{L^p}^{1/2}
       \left(\int_M\Abracket{Af,f}\dv_g\right)^{1/2}.
    \]
    The Fredholm alternative gives a unique solution of $\D_g u=Af$ with $u\perp_{L^2}\ker\D_g$. Since
    \[
       W^{1,2p/(p+1)}\Subset L^{2p/(p-1)}\hookrightarrow L^2(M,\Sigma_gM;A),
    \]
    the map $f\mapsto u-P_0(A)u$ is compact on the weighted kernel complement. Green's identity shows it is self-adjoint and injective. Its inverse, together with zero on the kernel, is the stated realization. Thus its spectrum is real and discrete, with a complete weighted orthonormal eigenbasis. The equation and elliptic bootstrapping give $W^{1,p}$ regularity for each eigenspinor.

    Both signs occur infinitely often. Indeed, finite spans of positive or negative unweighted eigenspinors have respectively positive or negative Dirac quadratic form. Their Dirac derivatives can be truncated to $\{j^{-1}\id\le A\le j\id\}$ and made orthogonal to $\ker\D_g$ by subtracting a linear combination of the sections $A\theta$, with $\theta$ harmonic. The correcting coefficients tend to zero. The resulting Dirac solutions have finite reciprocal energy and converge in $W^{1,2p/(p+1)}$ to the original spinors. The quadratic form is continuous in this topology, so its positive and negative indices are both infinite. The spectral theorem gives eigenvalues diverging to both $+\infty$ and $-\infty$.

    We will be concerned with the continuous dependence on the weight~$A$, thus let us denote the associated eigenvalues resp. eigenspinors by~$\lambda_k(A)$ resp.~$\varphi_k(A)$.
    As in the classical unweighted case, we list the nonzero eigenvalues, counted with multiplicities, in an increasing manner
    \begin{align}
        -\infty\leftarrow \cdots \leq \lambda_{-2}(A)\leq \lambda_{-1}(A) < 0< \lambda_1(A) \leq \lambda_2(A)\le \cdots \to +\infty.
    \end{align}
    The kernel of the Dirac operator, ~$\ker\D_g$, is independent of~$A$ and may be zero.
    Denote~$h_0\coloneqq \dim\ker\D_g$.
    In case~$h_0\neq 0$, we denote a basis for~$\ker\D_g$ by~$\theta_j$,~$j = 1,\cdots,h_0$.
    Moreover, we assume that these eigenspinors satisfy
    \begin{equation}\label{eq:A-orthonormal}
        \begin{aligned}
        &\int_M \Abracket{A\varphi_j(A),\varphi_k(A)}_{g^s} \dv_g=\delta_{jk}, \quad \forall j,k\in \mathbb{Z}\setminus\braces{0},  \\
        &\int_M \Abracket{A\theta_j,\theta_k}_{g^s}\dv_g=\delta_{jk}, \forall j,k\in \braces{1,2,\cdots, h_0}, \\
        &\int_M \Abracket{A\varphi_k(A),\theta_j}_{g^s}\dv_g=0,\quad \forall k\in\mathbb{Z}\setminus\braces{0}, \; \forall j\in \braces{1,\cdots, h_0}.
    \end{aligned}
    \end{equation}
    
    These can be achieved via Gram-Schmidt orthogonalization procedure, noting that~$A$ is symmetric.
    They form a complete orthonormal basis for the weighted space $L^2(M,\Sigma_gM;A)$, whose norm is $\left(\int_M\Abracket{A\psi,\psi}\dv_g\right)^{1/2}$.
    Indeed, any spinor field~$\psi$ in this weighted space admits a unique decomposition
    \begin{align}
         \psi= \sum_{k\neq 0} a_k\,\varphi_k(A) + \sum_{j=1}^{h_0} b_j\,\theta_j
    \end{align}
    and~$\D_g$ acts on such spinors formally as
    \begin{align}
        \D_g\psi=\sum_{k\neq 0} \lambda_k(A) \; a_k \; A\varphi_k(A).
    \end{align}
    When~$A=\id_{\Sigma_g M}$, this reduces to the classical Fourier expansion of a spinor via unweighted eigenspinors.
    For $A=\id$, a spinor with regularity~$H^{s}$ (for~$s>0$) can be defined via the unweighted eigenspinor expansions~\cite{Ammann2009smallest}, which is readily seen to be equivalently described by
    \begin{align}
        \psi\in H^{s} \Longleftrightarrow \sum_{k\neq 0} |\lambda_k|^{2s} |a_k|^2 +\sum_{j}|b_j|^2<+\infty.
    \end{align}

    The spinor bundles have a complex structure in dimensions~$n\equiv 1,5\pmod{8}$, and have a quaternionic structure in dimensions~$n\equiv 2,3,4\pmod{8}$, see~\cite[Proposition II.3.10]{LawsonMichelsohn1989spin}.
    These symmetries are inherited by the weighted problem when $A$ commutes with them. Thus it is common that the real multiplicities of the eigenvalues of the Dirac operator are greater than one.
    Even modulo the complex or quaternionic structures, it is still not easy to get the multiplicities to one generically, see~\cite{Dahl2003Dirac}.
    This is in contrast to the Laplacian operator on Riemannian manifolds, whose eigenvalues are generically simple~\cite{Uhlenbeck1976generic}.
    We have to take care of the multiplicities in our treatment of the convergence problems.

    \

    The strong indefiniteness of the Dirac operator is the root for many technical problems in geometry and analysis.
    The classical min-max description of eigenvalues usually fails since the operator is neither bounded from above nor from below.
    In some cases where the spectrum of~$\D_g$ is symmetric with respect to the origin, which is the case for~$n\not\equiv 3\pmod{4}$, one can consider the nonnegative operator~$\D_g^2$, which is sometimes called Dirac-Laplacian, and extract the eigenvalues~$\lambda_k$ of~$\D_g$ from the eigenvalues~$\lambda_k^2$ of~$\D_g^2$, see e.g.~\cite[Chapter 5]{Ginoux2009Dirac}.
    This approach unfortunately fails for weighted eigenvalues.
    There is another characterization of the positive Dirac eigenvalues using the notion of \emph{Dirac sea} which refers to the subspaces spanned by such eigenspinors of negative eigenvalues.
    One has to restrict to the orthogonal complement of this Dirac sea~\cite{Dolbeault2000variational}.
    This is however not explicit and again not convenient for our weighted eigenvalue problems.

    In~\cite{Ammann2009smallest} B. Ammann used the dual variational principle to obtain another min-max description of the first positive eigenvalue of Dirac operator:
    \begin{align}
        \frac{1}{\lambda_1(g)}
        =\sup\braces{\frac{\int_M \Abracket{\psi,\D_g\psi}\dv_g }{\int_M \Abracket{\D_g\psi, \D_g\psi}\dv_g} \mid \psi\in H^1(\Sigma_g M)\setminus \Ker\D_g}.
    \end{align}
    This is an intuitive formulation, which will be generalized soon to describe all the eigenvalues of Dirac operators, in particular for the weighted ones.
    We remark that this formulation is related to several variational problems, among them we mention the spinorial Yamabe problem which is about the extremal of the normalized Dirac eigenvalues in the conformal class.

    \

    All trial spinors in the reciprocal quotients below belong to $H^1$ and satisfy
    \[
        \int_M\Abracket{A^{-1}\D_g\psi,\D_g\psi}\dv_g<\infty.
    \]
    The trial subspaces consist of such spinors. This restriction is automatic for smooth positive weights and for finite linear combinations of eigenspinors. In general it does not follow merely from $\psi\in H^1$. For an admissible expansion $\psi=\sum a_j\varphi_j(A)+\sum b_j\theta_j$, we have
    \[
    \int_M\Abracket{\psi,\D_g\psi}\dv_g=\sum_{j\ne0}\lambda_j(A)a_j^2,
    \qquad
    \int_M\Abracket{A^{-1}\D_g\psi,\D_g\psi}\dv_g=\sum_{j\ne0}\lambda_j(A)^2a_j^2.
    \]

    The positive weighted eigenvalues~$\lambda_k(A)$ can be similarly obtained using a min-max principle.
    Let~$Gr^*_{k-1}(A)$ denote the Grassmann manifold of~$(k-1)$-planes in~$H^{1}(\Sigma_g M)$ which are perpendicular to~$\ker\D_g$ with respect to the weighted inner product induced by~$A$, namely
    \begin{align}
        Gr^*_{k-1}(A)\coloneqq \braces{V\subset H^{1}(\Sigma_g M)\mid V \mbox{ is a subspace, } V\perp_A \Ker(\D_g), \;\dim V=k-1 }.
    \end{align}
    Here~$V\perp_A \ker\D_g$ means, for any~$\psi\in V$ and any~$\theta\in \ker\D_g$, we have
    \begin{align}
        \int_M \Abracket{A\psi,\theta}_{g^s}\dv_g =0.
    \end{align}

    \begin{prop}\label{prop:minmax}
        The~$k$-th positive weighted eigenvalue~$\lambda_k(A)$ is characterized by
        \begin{align}\label{eq:minmax positive}
        \frac{1}{\lambda_k(A)}
        =& \inf_{V\in Gr^*_{k-1}(A)} \; \sup_{0\neq \psi\perp_A (V\oplus \ker\D_g)}\braces{\frac{\int_M \Abracket{\psi,\D_g\psi}\dv_g }{\int_M \Abracket{A^{-1}\D_g\psi, \D_g\psi}\dv_g} }.
        \end{align}
    \end{prop}
    \begin{proof}
        We temporarily denote the right hand side above by~$\nu_k$ and aim to prove~$\frac{1}{\lambda_k(A)}=\nu_k$.

        Let~$V\in Gr_{k-1}^*$.
        Then in the~$k$-dimensional space~$\Span\braces{\varphi_1(A),\cdots,\varphi_k(A)}$ which is automatically~$A$-orthogonal to~$\ker\D_g$, there exists an element~$\psi_k$ which is orthogonal to~$V$ in the sense that
        \begin{align}
            \int_M \Abracket{A\psi_k,\phi}_{g^s}\dv_g=0, \qquad \forall \phi\in V\oplus\ker\D_g.
        \end{align}
        Writing~$\psi_k=\sum_{j=1}^k a_j\, \varphi_j(A)$, we see that
        \begin{align}
            \sup_{0\neq \psi\perp_A(V\oplus\ker\D_g)}\braces{\frac{\int_M \Abracket{\psi,\D_g\psi}\dv_g }{\int_M \Abracket{A^{-1}\D_g\psi, \D_g\psi}\dv_g} }
            &\geq \frac{\int_M \Abracket{\psi_k,\D_g\psi_k}\dv_g }{\int_M \Abracket{A^{-1}\D_g\psi_k, \D_g\psi_k}\dv_g}
            \\ &=\frac{\sum_{j=1}^k \lambda_j(A)\, a_j^2}{\sum_{j=1}^k \lambda_j(A)^2 a_j^2}
            \geq \frac{1}{\lambda_k(A)}.
        \end{align}
        Hence~$\nu_k\geq \frac{1}{\lambda_k(A)}$.

        By taking~$V_*=\Span\braces{\varphi_1(A),\cdots,\varphi_{k-1}(A)}$, every admissible $\psi\perp_A(V_*\oplus\ker\D_g)$ has only terms with $j\ge k$ or $j<0$ in its expansion. For both sets of indices,
        $\lambda_j(A)\le\lambda_j(A)^2/\lambda_k(A)$.
        Hence the quotient is at most $1/\lambda_k(A)$, and equality is attained at $\psi=\varphi_k(A)$. We see that
        \begin{align}
            \sup_{0\neq \psi\perp_A (V_*\oplus\ker\D_g)}\braces{\frac{\int_M \Abracket{\psi,\D_g\psi}\dv_g }{\int_M \Abracket{A^{-1}\D_g\psi, \D_g\psi}\dv_g} } = \frac{1}{\lambda_k(A)}.
        \end{align}
        Thus~$\nu_k\leq \frac{1}{\lambda_k(A)}$.
    \end{proof}
    In particular, we see from the proof that the infimum is attained by taking
    \begin{align}
        V_{k-1}(A)=\Span\braces{\varphi_1(A),\cdots,\varphi_{k-1}(A)}.
    \end{align}
    This fact will be used in the sequel.

    \begin{rmk}
        A similar statement holds for the negative eigenvalues:
        \begin{align}\label{eq:minmax negative}
        \frac{1}{\lambda_{-k}(A)}
        =& \sup_{V \in Gr^*_{k-1}(A)} \; \inf_{0\neq \psi\perp_A (V\oplus\ker\D_g)}\braces{\frac{\int_M \Abracket{\psi,\D_g\psi}\dv_g }{\int_M \Abracket{A^{-1}\D_g\psi, \D_g\psi}\dv_g} }.
        \end{align}
    \end{rmk}

    \begin{rmk}\label{rmk:Laplacian}
        There is a similar min-max characterization of the weighted eigenvalues for Laplacian operators, either on closed manifolds:
        \begin{align}
            -\Delta_g  u_k(a) =\lambda_k(a) \, a\, u_k(a),\qquad \mbox{ in } M.
        \end{align}
        where~$a \in C^\infty(M)$ is a positive weight function.
        The kernel consists of functions which are constant on each connected component; the trial functions below belong to $H^2(M)$.
        Then the positive weighted eigenvalues are characterized by
        \begin{align}
            \frac{1}{\lambda_k(a)}
            =\inf_{V\in Gr^*_{k-1}(a)} \sup_{0\neq u\perp_a (\ker\Delta_g\oplus V)} \braces{ \frac{\int_M |\nabla u|^2\dv_g }{ \int_M a^{-1}|\Delta_g u|^2 \dv_g } }.
        \end{align}
        This is not as convenient as the classical one.
    \end{rmk}

    \

    A similar argument gives another equivalent min-max characterization of the weighted eigenvalues.
    \begin{prop}\label{prop:another minmax}
        The $k$-th positive weighted eigenvalue~$\lambda_k(A)$ is characterized by
        \begin{align}
            \frac{1}{\lambda_k(A)}
            = \sup_{W\in Gr^*_k(A)} \; \inf_{0\neq\psi\perp_A\ker\D_g, \psi\in W}
            \braces{\frac{\int_M \Abracket{\psi,\D_g\psi}\dv_g }{\int_M \Abracket{A^{-1}\D_g\psi, \D_g\psi}\dv_g} },
        \end{align}
        and the~$k$-th negative weighted eigenvalue~$\lambda_{-k}(A)$ is characterized by
        \begin{align}
            \frac{1}{\lambda_{-k}(A)}
            = \inf_{W\in Gr^*_k(A)} \; \sup_{0\neq\psi\perp_A\ker\D_g, \psi\in W}
            \braces{\frac{\int_M \Abracket{\psi,\D_g\psi}\dv_g }{\int_M \Abracket{A^{-1}\D_g\psi, \D_g\psi}\dv_g} }.
        \end{align}
    \end{prop}

    \begin{proof}
        For the positive eigenvalue, the space $\Span\{\varphi_1(A),\ldots,\varphi_k(A)\}$ gives an infimum equal to $1/\lambda_k(A)$. Conversely, any admissible $k$-plane contains a nonzero spinor $A$-orthogonal to the first $k-1$ positive eigenspinors. The termwise estimate in the proof of Proposition~\ref{prop:minmax} bounds its quotient above by $1/\lambda_k(A)$. This proves the supremum--infimum formula. Applying it to $-\D_g$ proves the negative formula.

        Equivalently, the extrema may be taken over all admissible $k$-planes $W$ satisfying $W\cap\ker\D_g=\{0\}$. Subtracting $P_0(A)$ preserves both quadratic forms and the dimension of such a plane. This also permits the use of ordinary $L^2$-orthogonality to the kernel.
    \end{proof}

    As an application of the above minmax characterization of eigenvalues, we obtain a comparison principle for such weighted eigenvalues of the Dirac operator.

    \begin{proof}[Proof of Theorem~\ref{thm:comparison}]
        This follows from Proposition~\ref{prop:another minmax}. For $k>0$, take
        \[
            W=\Span\{\varphi_1(A_2),\ldots,\varphi_k(A_2)\}.
        \]
        This space intersects $\ker\D_g$ trivially, and the numerator of the quotient is positive for every nonzero~$\psi\in W$. Since $A_1^{-1}\le A_2^{-1}$, the weighted Rayleigh quotients satisfy
        \begin{align}
            \frac{\int_M\Abracket{\psi,\D_g\psi}\dv_g}{\int_M\Abracket{A_1^{-1}\D_g\psi,\D_g\psi}\dv_g}
            \ge
            \frac{\int_M\Abracket{\psi,\D_g\psi}\dv_g}{\int_M\Abracket{A_2^{-1}\D_g\psi,\D_g\psi}\dv_g}.
        \end{align}
        Taking the infimum over $0\ne\psi\in W$ and then using the supremum over admissible $k$-planes, we obtain
        \[
           \frac1{\lambda_k(A_1)}\ge\frac1{\lambda_k(A_2)},
        \]
        and hence $\lambda_k(A_1)\le\lambda_k(A_2)$, as desired. Applying the same argument to $-\D_g$ gives the reversed inequality for $k<0$.
    \end{proof}

    \begin{rmk}
    In the equivalent formulation using~$W\cap\ker\D_g=\{0\}$, the kernel condition is independent of the weight, although the finite reciprocal energy condition may still depend on $A$.
    However, $A_1\ge A_2$ implies $A_1^{-1}\le A_2^{-1}$, so every such trial space admissible for $A_2$ is also admissible for $A_1$.
    The comparison above uses a space on which the numerator is positive; a quotient inequality for arbitrary spinors would not be valid.
    \end{rmk}

    As a consequence of the above comparison principle, for each fixed smooth positive weight the weighted eigenvalues have the same order of growth as the unweighted ones as~$k\to \pm\infty$. The leading constants need not coincide, since multiplication of $A$ by a positive constant divides every eigenvalue by that constant.

    In general we cannot expect for a rigidity statement in Theorem~\ref{thm:comparison} of the form that~$\lambda_k(A_1)=\lambda_k(A_2)$ forces~$A_1=A_2$.
    A weaker version is given as following. 
    \begin{prop}
        Let~$A_1,A_2\in \End(\Sigma_g M)$ be~$L^p$ symmetric, positive definite a.e. with $A_1\geq A_2$.

        If~$\lambda_k(A_1)=\lambda_k(A_2)$ for some~$k\neq 0$, then there exists a nonzero spinor~$\psi$ such that 
        \begin{align}
            (A_1-A_2)\psi=0, \qquad \mbox{a.e.}
        \end{align}
    \end{prop}
    \begin{proof}
        Suppose~$k>0$ and consider the space~$\Span\braces{\varphi_1(A_2),\cdots,\varphi_k(A_2)}$.
        The equality~$\lambda_k(A_1)=\lambda_k(A_2)$ implies that for some~$\psi\neq 0$,
        \begin{align}
            \D\psi=\lambda_k(A_2)A_2\psi 
        \end{align}
        while for~$A_1$ this forces 
        \begin{align}
            \int_{M}\Abracket{(A_1^{-1}-A_2^{-1})\D\psi,\D\psi}\dv_g=0. 
        \end{align}
        Since~$A_1^{-1}-A_2^{-1}\leq 0$, this implies~$(A_1^{-1}-A_2^{-1})A_2\psi =0$ a.e., that is, ~$(A_1-A_2)\psi=0$ a.e.

        The case~$k<0$ can be handled with in a similar way.
    \end{proof}

    Since the rank of the spinor bundle is greater than one, one can construct examples where~$A_1-A_2\gneqq 0$ but has a nontrivial kernel. 
    In the case of scalar weights, by unique continuation properties for solutions of linear Dirac equations (see e.g.~\cite{BorrelliWu2026Nodal,BBMW2002weakUCP}), this does forces the scalar weights to coincide. 
    \begin{cor}\label{cor:rigidity-scalar}
        Let~$a_1,a_2\in C^0(M,\R)$ be positive a.e. and~$a_1\geq a_2$.  
        If~$\lambda_k(a_1)=\lambda_k(a_2)$ for some~$k\neq 0$, then~$a_1\equiv a_2$. 
    \end{cor}


\section{A priori estimates for the weighted eigenspinors}\label{sect:estimates}

    In this section we establish a priori estimates of the~$H^1$ norms and the~$C^\alpha$ norms of weighted eigenspinors.
    These will be used in the proof of the complete continuity of eigenpairs.
    The crucial point here is the dependence on the weights and on the eigenvalues.

    For the moment we omit the index~$k$ and consider a solution~$\varphi$ of the equation
    \begin{align}\label{eq:eigenspinor}
        \D\varphi= \lambda \, A \, \varphi \qquad \mbox{ in } M
    \end{align}
    and a normalization condition
    \begin{equation}\label{eq:unit norm}
        1=\int_M \Abracket{A\,\varphi,\varphi}\dv_g =\int_M |\sqrt{A}\,\varphi|^2 \dv_g.
    \end{equation}
    Here~$\sqrt{A}$ stands for the unique symmetric positive definite square root of~$A$.
    By the standard elliptic regularity theory we get the following estimates in terms of the eigenvalue~$\lambda$ and the weight~$A$.

    \begin{prop}\label{prop:regularity}
        Let~$\varphi\in H^{1/2}(\Sigma_g M)$ be a solution of~\eqref{eq:eigenspinor} where~$A\in L^p(M,\End(\Sigma_g M))$ for some~$p>n$, and $\lambda\ne0$, with normalization~\eqref{eq:unit norm}. Choose $c>0$ such that\begin{equation}\label{eq:harmonic lower bound}\int_M\Abracket{Ah,h}\dv_g\ge c\|h\|_{L^2}^2\quad\text{for every }h\in\ker\D_g.\end{equation}
        There exist~$\alpha=\alpha(n,p)\in (0,1)$ and~$C=C(n,p,g)$ such that
        \begin{align}
            \|\varphi\|_{W^{1,2}(M)}\leq C(n,p,g) \parenthesis{1+|\lambda|\|A\|_{L^p} }^{T_1}
            \parenthesis{ |\lambda|\|A\|_{L^p}^{\frac12}\left(1+c^{-1/2}\|A\|_{L^p}^{\frac12}\right) }
        \end{align}
        and
        \begin{align}
            \|\varphi\|_{C^\alpha(M)}\leq C(n,p,g) \parenthesis{1+|\lambda|\|A\|_{L^p} }^{T_2}
            \parenthesis{ |\lambda|\|A\|_{L^p}^{\frac12}\left(1+c^{-1/2}\|A\|_{L^p}^{\frac12}\right) },
        \end{align}
        where
        \begin{align}
            T_1=\left\lceil\frac{n}{2(p-n)}\right\rceil, & &
            T_2=  \left\lfloor\frac{n(p-1)}{2(p-n)}\right\rfloor.
        \end{align}
        If $\ker\D_g=0$, the factor $1+c^{-1/2}\|A\|_{L^p}^{1/2}$ is omitted. The H\"older exponent can be chosen so that the same argument gives a bound in $C^{2\alpha}$.
    \end{prop}

    \begin{proof}
        Since~$\D$ is elliptic, we know that for any~$q>1$,
        \begin{align}
            \|\varphi\|_{W^{1,q}(M)} \leq C(q,g)\parenthesis{ \|\varphi\|_{L^q}+ \|\D\varphi\|_{L^q} }
        \end{align}
        see e.g.~\cite{Ammann2003Habil}.
        Repeated applications of such estimates would give the desired results.

        To start with, we write $q_1=2p/(p+1)$ and note that
        \begin{align}
            \|\D_g\varphi\|_{L^{q_1}}
            =\|\lambda A\varphi\|_{L^{q_1}}
            \le|\lambda|\|\sqrt A\|_{L^{2p}}\|\sqrt A\varphi\|_{L^2}
            \le|\lambda|\|A\|_{L^p}^{1/2}.
        \end{align}
        Let $\varphi_0=P_0(\id)\varphi$ be the ordinary harmonic projection. The elliptic estimate on the kernel complement gives
        \[
           \|\varphi-\varphi_0\|_{W^{1,q_1}}
           \le C\|\D_g\varphi\|_{L^{q_1}}.
        \]
        Since $\lambda\ne0$, the equation implies $\varphi\perp_A\ker\D_g$. Weighted Cauchy--Schwarz and~\eqref{eq:harmonic lower bound} give
        \begin{align}
        c\|\varphi_0\|_{L^2}^2
        &\le\int_M\Abracket{A\varphi_0,\varphi_0}\dv_g
        \le\int_M\Abracket{A(\varphi-\varphi_0),\varphi-\varphi_0}\dv_g\\
        &\le C\|A\|_{L^p}\|\varphi-\varphi_0\|_{W^{1,q_1}}^2.
        \end{align}
        Here $W^{1,q_1}\hookrightarrow L^{2p/(p-1)}$, since $p>n$. Norm equivalence on the fixed harmonic space therefore yields
        \begin{align}
            \|\varphi\|_{W^{1,q_1}}
            \le C(n,p,g)|\lambda|\|A\|_{L^p}^{1/2}
                 \left(1+c^{-1/2}\|A\|_{L^p}^{1/2}\right).
        \end{align}
        When the kernel is trivial, the same estimate holds with the parenthetical factor omitted.

        The Sobolev embedding theorem implies that
        \begin{align}
            \|\varphi\|_{L^{r_1}}\leq C(n,q_1,g)\|\varphi\|_{W^{1,q_1}}
            \leq C(n,p,g) \parenthesis{ |\lambda|\|A\|_{L^p}^{\frac12}\left(1+c^{-1/2}\|A\|_{L^p}^{\frac12}\right) },
        \end{align}
        where~$\frac{1}{r_1}=\frac{1}{q_1}-\frac{1}{n}$.

        Let~$1<q_2<r_1$ be such that~$\frac{1}{q_2}=\frac{1}{p}+\frac{1}{r_1}$.
        Then
        \begin{align}
            \|\varphi\|_{W^{1,q_2}}
            \leq& C(q_2,g)\parenthesis{\|\varphi\|_{L^{q_2}} +|\lambda| \|A\;\varphi\|_{L^{q_2}}}
            \leq C(q_2,g)\parenthesis{ C(p,g)+ |\lambda| \|A\|_{L^p} } \|\varphi\|_{L^{r_1}} \\
            \leq& C(q_2,n,p,g)\parenthesis{1+|\lambda|\|A\|_{L^p}}\parenthesis{ |\lambda|\|A\|_{L^p}^{\frac12}\left(1+c^{-1/2}\|A\|_{L^p}^{\frac12}\right)  } .
        \end{align}
        If $q_2<n$, then for~$\frac{1}{r_2}=\frac{1}{q_2}-\frac{1}{n}$, we have
        \begin{align}
            \|\varphi\|_{L^{r_2}}\leq C(q_2,n,g)\|\varphi\|_{W^{1,q_2}}
            \leq  C(q_2,n,p,g)\parenthesis{1+|\lambda|\|A\|_{L^p}}\parenthesis{ |\lambda|\|A\|_{L^p}^{\frac12}\left(1+c^{-1/2}\|A\|_{L^p}^{\frac12}\right)  } .
        \end{align}

        Iterating this procedure~$j\geq 1$ times, we obtain
        \begin{align}
            \|\varphi\|_{W^{1,q_j}}\leq C(q_j,n,p,g)\parenthesis{1+|\lambda|\|A\|_{L^p} }^{j-1}
            \parenthesis{ |\lambda|\|A\|_{L^p}^{\frac12}\left(1+c^{-1/2}\|A\|_{L^p}^{\frac12}\right) },
        \end{align}
        with
        \begin{align}
            \frac{1}{q_j}=\frac{1}{q_1} - \parenthesis{\frac{1}{n}-\frac{1}{p}}(j-1).
        \end{align}
        This recurrence is used while $q_j<n$. Since $p>n$, the exponents increase. If $q_j=n$, use $W^{1,n}\hookrightarrow L^r$ for a sufficiently large finite $r$ to reach an exponent greater than $n$ at the next step. Once $q_j>n$, the spinor is bounded, and the equation gives $\varphi\in W^{1,p}$.

        (1) To get the~$W^{1,2}$ estimate we take~$j_1$ such that~$q_{j_1}\geq 2$, say,
        \begin{align}
            j_1 = \left\lceil\frac{n}{2(p-n)}\right\rceil+1\equiv T_1 +1.
        \end{align}
        Then
        \begin{align}
            \|\varphi\|_{W^{1,2}} \leq C(n,p,g) \parenthesis{1+|\lambda|\|A\|_{L^p} }^{T_1}
            \parenthesis{ |\lambda|\|A\|_{L^p}^{\frac12}\left(1+c^{-1/2}\|A\|_{L^p}^{\frac12}\right) }.
        \end{align}

        (2) To get a H\"older estimate, we take~$j_2$ such that~$q_{j_2}>n$, for example,
        \begin{align}
            j_2=  \left\lfloor\frac{n(p-1)}{2(p-n)}\right\rfloor+1\equiv T_2 + 1.
        \end{align}
        Then by the Sobolev embedding
        \begin{align}
            W^{1,q_{j_2}}\hookrightarrow C^{1-\frac{n}{q_{j_2}}}\hookrightarrow C^{2\alpha},\qquad 0<2\alpha<1-\frac{n}{q_{j_2}},
        \end{align}
        we see that~$\varphi\in C^\alpha$, with
        \begin{align}
            \|\varphi\|_{C^\alpha} \leq C(n,p,g) \parenthesis{1+|\lambda|\|A\|_{L^p} }^{T_2}
            \parenthesis{ |\lambda|\|A\|_{L^p}^{\frac12}\left(1+c^{-1/2}\|A\|_{L^p}^{\frac12}\right) }.
        \end{align}
    \end{proof}

    \begin{rmk}
        The constant $C(n,p,g)$ is independent of $\lambda,A,\varphi$; the dependence on the weight is displayed explicitly. The number $c$ exists by positivity on the finite-dimensional space $\ker\D_g$. Under~\eqref{hypo}, it can be chosen uniformly for all $A_m$: the matrices $\left(\int_M\Abracket{A_m\theta_i,\theta_j}\dv_g\right)$ in a fixed harmonic basis converge to a positive definite matrix.
        This will be used in later applications.

        However, the above estimates may not be optimal.
        One may consider to use the Moser iteration argument, as did in~\cite{WenYangZhang2018complete} to get a better power of the factor~$(1+|\lambda|\|A\|_{L^p})$.

        When we get a uniform positive lower bound of~$|\lambda|\|A\|_{L^p}$, then we may also replace the factor~$(1+|\lambda|\|A\|_{L^p})$ by~$(|\lambda|\|A\|_{L^p})$.

        The H\"older index~$\alpha$ can be improved. But in general we cannot hope to get uniform~$C^\beta$ estimates for all~$\beta\in (0,1)$ under only the stated $L^p$ bound and the harmonic lower bound.
    \end{rmk}


\section{Continuity of weighted eigenvalues and eigenspinors of Dirac operator}\label{sect:closed}

    In this section we consider the continuous dependence of the eigenvalues and eigenspinors on the weight~$A$ and prove Theorem~\ref{thm:continuity-closed}.

    \begin{proof}[Proof of Theorem~\ref{thm:continuity-closed}]
        We divide the proof into the following steps.

        \

        \noindent{\bf Step I.} We show that for any~$k\geq 1$, the sequence~$\parenthesis{\lambda_k(A_m)}_{m\geq 1}$ is bounded.
        Since it is a positive sequence, it suffices to show that
        \begin{align}
            \limsup_{m\to+\infty} \lambda_k(A_m)\leq \lambda_k(A).
        \end{align}

        \

        For each $j=1,\ldots,k$, take the solution $\eta_j^{(m)}$ of
        \begin{align}
          \D_g\eta_j^{(m)}=\lambda_j(A)A_m\bigl(\varphi_j(A)-P_0(A_m)\varphi_j(A)\bigr),
          \qquad \eta_j^{(m)}\perp_{A_m}\ker\D_g.
        \end{align}
        The right-hand side is orthogonal to the kernel, so this equation is solvable. To verify convergence, use a fixed harmonic basis $\theta_1,\ldots,\theta_{h_0}$. The coefficients of $P_0(A_m)u=\sum_a b_a^{(m)}\theta_a$ solve
        \begin{align}\label{eq:harmonic projection coefficients}
        \sum_a b_a^{(m)}\int_M\Abracket{A_m\theta_a,\theta_b}\dv_g
           =\int_M\Abracket{A_m u,\theta_b}\dv_g.
        \end{align}
        The coefficient matrices converge to a positive definite matrix. Thus $P_0(A_m)\varphi_j(A)\to0$ in every smooth norm. The fixed elliptic estimate, followed by the harmonic correction above, bounds $\eta_j^{(m)}$ in $W^{1,p}$. Every subsequential limit solves the limiting equation and is $A$-orthogonal to the kernel; hence it is $\varphi_j(A)$. Consequently,
        \[
        \eta_j^{(m)}\to\varphi_j(A)\quad\text{uniformly and in }H^s,\quad0<s<1,
        \qquad\eta_j^{(m)}\rightharpoonup\varphi_j(A)\quad\text{in }W^{1,p}.
        \]

        Consider the test spinors of the form
        \begin{align}
            \eta^{(m)}=\sum_{j=1}^k a_j^{(m)}\eta_j^{(m)},
        \end{align}
        where the coefficients are chosen so that
        \begin{align}
            \int_M\Abracket{A_m\eta^{(m)},\varphi_i(A_m)}\dv_g=0\quad(1\le i\le k-1),
            \qquad\sum_{j=1}^k[a_j^{(m)}]^2=1.
        \end{align}
        The $k$ test spinors are linearly independent for large $m$, by their convergence to $\varphi_1(A),\ldots,\varphi_k(A)$. They are $A_m$-orthogonal to the kernel and have finite reciprocal energy. Thus
        \begin{align}
            \frac1{\lambda_k(A_m)}\ge
            \frac{\int_M\Abracket{\eta^{(m)},\D_g\eta^{(m)}}\dv_g}
                 {\int_M\Abracket{A_m^{-1}\D_g\eta^{(m)},\D_g\eta^{(m)}}\dv_g}.
        \end{align}
        On one hand, weak convergence of the weights and uniform convergence of the test spinors give
        \[
        \int_M\Abracket{\eta_i^{(m)},\D_g\eta_j^{(m)}}\dv_g
        =\lambda_j(A)\int_M\Abracket{\eta_i^{(m)},A_m(\varphi_j(A)-P_0(A_m)\varphi_j(A))}\dv_g
        \longrightarrow\lambda_j(A)\delta_{ij}.
        \]
        On the other hand, the denominator satisfies the exact identity
        \begin{align}
        &\int_M\Abracket{A_m^{-1}\D_g\eta_i^{(m)},\D_g\eta_j^{(m)}}\dv_g\\
        &\quad=\lambda_i(A)\lambda_j(A)
        \int_M\Abracket{A_m(\varphi_i(A)-P_0(A_m)\varphi_i(A)),
                                \varphi_j(A)-P_0(A_m)\varphi_j(A)}\dv_g
        \longrightarrow\lambda_j(A)^2\delta_{ij}.
        \end{align}
        Since only finitely many coefficients occur, these convergences are uniform on their unit sphere. Consequently,
        \[
        \frac1{\lambda_k(A_m)}\ge
        \frac{\sum_{j=1}^k\lambda_j(A)[a_j^{(m)}]^2+o(1)}
             {\sum_{j=1}^k\lambda_j(A)^2[a_j^{(m)}]^2+o(1)}
        \ge\frac1{\lambda_k(A)}-o(1).
        \]
        It follows that
        \[
        \liminf_{m\to\infty}\frac1{\lambda_k(A_m)}\ge\frac1{\lambda_k(A)},
        \qquad
        \limsup_{m\to\infty}\lambda_k(A_m)\le\lambda_k(A).
        \]

    \

    \noindent{\bf Step II.} For a fixed~$k\geq 1$, we show that~$\{\varphi_k(A_m)\}_{m\geq 1}$ is a bounded sequence in~$H^1(M)$ and~$C^\alpha(M)$ for some~$\alpha=\alpha(n,p)\in(0,1)$.

    \

    Note that each~$\varphi_k(A_m)$ is a solution to a linear equation
    \begin{align}
        \D_g\varphi_k(A_m)=\lambda_k(A_m)\, A_m \, \varphi_k(A_m), \qquad \mbox{ in } M
    \end{align}
    and satisfies a normalization condition
    \begin{align}
        1=\int_M \Abracket{A_m\; \varphi_k(A_m),\varphi_k(A_m)}\dv_g=\int_M |\sqrt{A_m}\;\varphi_k(A_m)|^2\dv_g.
    \end{align}
    Here~$\sqrt{A_m}$ stands for the symmetric positive definite square root matrix of~$A_m$.
    Applying Proposition~\ref{prop:regularity} we get estimates on the~$H^1$ norm and~$C^\alpha$ norms in terms of~$\lambda_k(A_m)$, $\|A_m\|_{L^p}$, and the constant in~\eqref{eq:harmonic lower bound}.
    We have seen in the last step that~$\lambda_k(A_m)$ are uniformly bounded.
    As for the weights, weak $L^p$ convergence gives a uniform bound for $\|A_m\|_{L^p}$, and the remark after Proposition~\ref{prop:regularity} gives a common positive constant in~\eqref{eq:harmonic lower bound}.
    Thus,~$\{\varphi_k(A_m)\}_{m\geq 1}$ is bounded in~$H^1$ as well as in~$C^{2\alpha}$. The final bootstrap also gives a uniform $W^{1,p}$ bound.

    \

    \noindent{\bf Step III.} Having shown the boundedness of~$\braces{\lambda_k(A_m)\mid m\geq 1}$ and~$\braces{\varphi_k(A_m)\mid m\geq 1}$, we can extract a subsequence~$\{m_l\}_{l\geq 1}$ such that
    \begin{align}
        \lim_{l\to+\infty}\lambda_k(A_{m_l})= \lambda_\infty,
    \end{align}
    and
    \begin{align}
        \varphi_k(A_{m_l}) \to \varphi_\infty \quad\mbox{in } H^s(M)\cap C^\alpha(M),\quad 0<s<1,\qquad \varphi_k(A_{m_l})\rightharpoonup\varphi_\infty\ \mbox{in }W^{1,p}.
    \end{align}
    Then in this step we show that~$\lambda_\infty=\lambda_k(A)$ and~$\varphi_\infty\in \Eigen(A^{-1}\D_g; \lambda_k(A))$.

    \

    For any smooth test spinor~$\eta$, we have
    \begin{align}
        \int_M \Abracket{\D_g\varphi_\infty,\eta}\dv_g
        =& \lim_{l\to+\infty} \int_M \Abracket{\D_g \varphi_k(A_{m_l}),\eta}\dv_g \\
        =&\lim_{l\to+\infty} \int_M \Abracket{\lambda_k(A_{m_l}) \,A_{m_l}\, \varphi_k(A_{m_l}),\eta}\dv_g \\
        =&\int_M \lambda_\infty \Abracket{A\, \varphi_\infty,\eta}\dv_g.
    \end{align}
    That is,~$\varphi_\infty\in H^1\cap C^\alpha$ is a weak solution of
    \begin{align}
        \D_g \varphi_\infty = \lambda_\infty \, A\, \varphi_\infty,
    \end{align}
    which satisfies
    \begin{align}
        \int_M \Abracket{A \varphi_\infty, \, \varphi_\infty}\dv_g= \lim_{l\to+\infty}\int_M \Abracket{A_{m_l}\, \varphi_k(A_{m_l}), \, \varphi_k(A_{m_l})} \dv_g=1.
    \end{align}
    Thus~$\varphi_\infty$ is an eigenspinor of~$\D_g$ with weight~$A$ for the eigenvalue~$\lambda_\infty$.

    It remains to show~$\lambda_\infty=\lambda_k(A)$.
    We prove this by induction.

    For~$k=1$: we know that~$0\leq \lambda_\infty\leq \lambda_1(A)$. If~$\lambda_\infty\neq \lambda_1(A)$, then~$\lambda_\infty=0$ and~$\varphi_\infty$ is a harmonic spinor.
    This is impossible since, for any harmonic spinor~$\theta\in\ker\D_g$,
    \begin{align}
        \int_M \Abracket{A\, \varphi_\infty,\theta}\dv_g
        =\lim_{l\to+\infty}\int_M \Abracket{A_{m_l}\, \varphi_1(A_{m_l}),\theta} \dv_g=0, \quad
        \mbox{contradiction!}
    \end{align}
    Thus~$\lambda_\infty=\lambda_1(A)$ and~$\varphi_\infty\in \Eigen(A^{-1}\D_g, \lambda_1(A))$.

    Inductively suppose the statements hold for~$1,2,\cdots, k-1$.
    Pass to a common further subsequence for the first $k$ eigenspinors. Their limits are $A$-orthonormal, and the first $k-1$ limiting eigenvalues have the prescribed indices by induction. If $\lambda_\infty<\lambda_k(A)$, these $k$ limits would lie in the positive spectral subspace below $\lambda_k(A)$, which has dimension at most $k-1$. This is impossible. Hence $\lambda_\infty\geq\lambda_k(A)$ and it follows that
    \begin{align}
        \lambda_\infty=\lambda_k(A), & & \mbox{ and } & &
        \varphi_\infty\in \Eigen(A^{-1}\D_g,\lambda_k(A)).
    \end{align}

    \

    \noindent{\bf Step IV.} Now we show the convergence of the full sequences~$(\lambda_k(A_m) )$ to~$\lambda_k(A)$.
    This indeed follows from the well known fact that if any subsequence has a further sub-subsequence converging to a fixed limit, then this original sequence converges to the same limit.
    A proof can also be found in~\cite{WenYangZhang2018complete}.

    The convergence for the negative eigenvalues follows analogously. Applying the min--max characterization for the negative eigenvalues \eqref{eq:minmax negative} within the framework of \noindent{\bf Step I.}-\noindent{\bf Step IV.} yields the corresponding inequality
    \begin{align}
    \liminf_{m \to \infty} \lambda_k(A_m) \geq \lambda_k(A),
    \qquad \forall\, k<0,
    \end{align}
    thereby establishing the result for all $k \in \mathbb{Z}_*$.

    We should remark that the convergence of the sequence~$\varphi_k(A_m)$ in general fails, since the eigenvalues of Dirac operators in general have multiplicities greater than one, and one can easily construct examples where the original sequence~$\varphi_k(A_m)$ fails to converge.
    A possible remedy for this is by using the concept of eigen projectors.

    \

\noindent{\bf Step V.} We prove the continuity of the projectors. Fix $\ell\geq1$ and write $\sigma_0=0$.
 
As a first step, we show that,
for any $\psi\in H_\ell(A)$ with
$\int_M\Abracket{A\psi,\psi}\dv_g=1$, one has
$P_\ell^\perp(A_m)\psi\to0$ in $L^2$ as $m\to+\infty$.
The convergence in operator norm will be established afterwards.

Set $\mu=\mu_\ell(A)$ and $h_m=P_0(A_m)\psi$.
For the test spinor, solve
\begin{align}\label{eq:closed-projection-eta}
    \D_g\eta^{(m)}
    =\mu A_m(\psi-h_m),
    \qquad
    \eta^{(m)}\perp_{A_m}\ker\D_g.
\end{align}
The right-hand side is $L^2$-orthogonal to $\ker\D_g$,
so the equation is solvable, and the weighted orthogonality
condition determines $\eta^{(m)}$ uniquely.
As in Step I, $h_m\to0$ in every smooth norm.
The elliptic estimate and uniqueness of the limiting equation give
uniform $W^{1,p}$ bounds and
\begin{align}
    \eta^{(m)}\to\psi
    \quad\text{in }C^\alpha\text{ and }H^s,
    \qquad 0<s<1.
\end{align}
Its numerator satisfies
\begin{align}\label{eq:closed-projection-eta-num}
    \int_M\Abracket{\D_g\eta^{(m)},\eta^{(m)}}\dv_g
    =\mu\int_M
    \Abracket{A_m(\psi-h_m),\eta^{(m)}}\dv_g
    \longrightarrow\mu,
\end{align}
and its reciprocal energy equals
\begin{align}\label{eq:closed-projection-eta-denom}
    \int_M
    \Abracket{A_m^{-1}\D_g\eta^{(m)},\D_g\eta^{(m)}}\dv_g
    =\mu^2\int_M
    \Abracket{A_m(\psi-h_m),\psi-h_m}\dv_g\longrightarrow\mu^2.
\end{align}

For each $m\geq1$, expand $\eta^{(m)}$ as
\begin{align}
    \eta^{(m)}
    =\sum_{k\neq0}a_k^{(m)}\varphi_k(A_m).
\end{align}
There is no kernel component since
$\eta^{(m)}\perp_{A_m}\ker\D_g$.
Hypothesis~\eqref{hypo} implies
\begin{align}
    1
    &=\int_M\Abracket{A\psi,\psi}\dv_g
    =\lim_{m\to+\infty}
    \int_M\Abracket{A_m\eta^{(m)},\eta^{(m)}}\dv_g
    =\lim_{m\to+\infty}\sum_{k\neq0}|a_k^{(m)}|^2\\
    &=\lim_{m\to+\infty}\left(
    \|P_\ell(A_m)\eta^{(m)}\|_{L^2_{A_m}}^2
    +\sum_{1\leq k\leq\sigma_{\ell-1}}|a_k^{(m)}|^2
    +\sum_{k>\sigma_\ell}|a_k^{(m)}|^2
    +\sum_{k<0}|a_k^{(m)}|^2
    \right).
\end{align}
The first term is the squared weighted norm of the projection.

For each fixed lower positive index
$1\leq k\leq\sigma_{\ell-1}$,
Steps III--IV, the uniform convergence of $\eta^{(m)}$,
and the subsequence principle give
\begin{align}\label{eq:closed-lower-coefficients}
    a_k^{(m)}
    =\int_M
    \Abracket{A_m\eta^{(m)},\varphi_k(A_m)}\dv_g
    \longrightarrow0.
\end{align}
Indeed, along any common convergent subsequence, the limiting
eigenspinor is $A$-orthogonal to $\psi$, since its eigenvalue is smaller than $\mu_\ell(A)$. Consequently,
\begin{align}
    \lim_{m\to+\infty}
    \sum_{1\leq k\leq\sigma_{\ell-1}}|a_k^{(m)}|^2=0.
\end{align}

Define
\begin{align}
    \delta_m\coloneqq
    \sum_{1\leq k\leq\sigma_{\ell-1}}
    a_k^{(m)}\varphi_k(A_m),
\end{align}
and consider the spinor
\begin{align}
    \phi_m
    &\coloneqq\eta^{(m)}-\delta_m\\
    &=\sum_{\sigma_{\ell-1}<k\leq\sigma_\ell}
    a_k^{(m)}\varphi_k(A_m)
    +\sum_{k>\sigma_\ell}a_k^{(m)}\varphi_k(A_m)
    +\sum_{k<0}a_k^{(m)}\varphi_k(A_m).
\end{align}
The finitely many coefficients in $\delta_m$ tend to zero,
and the corresponding eigenspinors have uniform $W^{1,p}$ bounds
by Step II. Hence $\delta_m\to0$ in $W^{1,p}$, and its numerator
and reciprocal energy also tend to zero. In particular,
\begin{align}
    \lim_{m\to+\infty}
    \int_M\Abracket{A_m\phi_m,\phi_m}\dv_g=1,
\end{align}
and
\begin{align}
    \phi_m\perp_{A_m}\ker\D_g,\qquad
    \phi_m\perp_{A_m}
    \Span\{\varphi_1(A_m),\dots,
    \varphi_{\sigma_{\ell-1}}(A_m)\}.
\end{align}

By a direct calculation, one obtains
\begin{align}\label{eq:phi m D}
    \int_M\Abracket{\D_g\phi_m,\phi_m}\dv_g
    ={}&\sum_{\sigma_{\ell-1}<k\leq\sigma_\ell}
    \lambda_k(A_m)|a_k^{(m)}|^2
    +\sum_{k>\sigma_\ell}
    \lambda_k(A_m)|a_k^{(m)}|^2\\
    &\quad+\sum_{k<0}
    \lambda_k(A_m)|a_k^{(m)}|^2,
\end{align}
while
\begin{align}\label{eq:phi m DD}
    \int_M
    \Abracket{A_m^{-1}\D_g\phi_m,\D_g\phi_m}\dv_g
    ={}&\sum_{\sigma_{\ell-1}<k\leq\sigma_\ell}
    \lambda_k(A_m)^2|a_k^{(m)}|^2
    +\sum_{k>\sigma_\ell}
    \lambda_k(A_m)^2|a_k^{(m)}|^2\\
    &\quad+\sum_{k<0}
    \lambda_k(A_m)^2|a_k^{(m)}|^2.
\end{align}
Moreover, subtracting the low positive spectral part in
\eqref{eq:closed-projection-eta-num}--%
\eqref{eq:closed-projection-eta-denom} gives
\begin{equation}\label{eq:numtor mu l}
    \lim_{m\to+\infty}
    \int_M\Abracket{\D_g\phi_m,\phi_m}\dv_g
    =\mu_\ell(A),
\end{equation}
and
\begin{equation}\label{eq:denom mu l}
    \lim_{m\to+\infty}
    \int_M
    \Abracket{A_m^{-1}\D_g\phi_m,\D_g\phi_m}\dv_g
    =\mu_\ell(A)^2.
\end{equation}
The claim is that
\begin{align}
    \lim_{m\to+\infty}
    \left(
    \sum_{\sigma_\ell<k}|a_k^{(m)}|^2
    +\sum_{k<0}|a_k^{(m)}|^2
    \right)=0.
\end{align}
which would complete the proof of convergence in the weighted norm.

Indeed, for convenience, set
\begin{align}
  N_m\coloneqq\int_M\Abracket{\D_g\phi_m,\phi_m}\dv_g,
  \qquad
  D_m\coloneqq\int_M\Abracket{A_m^{-1}\D_g\phi_m,\D_g\phi_m}\dv_g.
\end{align}
In view of~\eqref{eq:numtor mu l} and~\eqref{eq:denom mu l}, $N_m\longrightarrow\mu$ and $D_m\longrightarrow\mu^2$ are satisfied. The index sets are defined by
\begin{align}
  I=\{\sigma_{\ell-1}<k\leq\sigma_\ell\},\qquad
  H=\{k>\sigma_\ell\},\qquad
  J=\{k<0\},
\end{align}
and set $g(t)\coloneqq\frac{t(t-\mu)}{\mu}$.

To compare the quotient $N_m/D_m$ with $1/\mu$. Substituting the expansions~\eqref{eq:phi m D} and~\eqref{eq:phi m DD} into the following equation, yields:
\begin{align}
  F_m&\coloneqq D_m\left(\frac1\mu-\frac{N_m}{D_m}\right)\\
  &=\frac{D_m}{\mu}-N_m\\
  &=\sum_{k\in I\cup H\cup J}
  \left(\frac{\lambda_k(A_m)^2}{\mu}-\lambda_k(A_m)\right)|a_k^{(m)}|^2\\
  &=\sum_{k\in I}g(\lambda_k(A_m))|a_k^{(m)}|^2
  +\sum_{k\in H}g(\lambda_k(A_m))|a_k^{(m)}|^2
  +\sum_{k\in J}g(\lambda_k(A_m))|a_k^{(m)}|^2.
\end{align}
The first summation part in the expansion of $F_m$ may be negative, so the convergence $F_m\to0$ alone does not imply the claim. We therefore estimate the three sums separately.

For the index in $I$, set
\begin{align}
  \varepsilon_m\coloneqq\max_{k\in I}|\lambda_k(A_m)-\mu|.
\end{align}
Since $I$ is finite and $\lambda_k(A_m)\to\mu$ for each $k\in I$, $\varepsilon_m\to0$ holds. Consequently,
\begin{align}
  \left|\sum_{k\in I}g(\lambda_k(A_m))|a_k^{(m)}|^2\right|
  \leq\frac{(\mu+\varepsilon_m)\varepsilon_m}{\mu}
  \sum_{k\in I}|a_k^{(m)}|^2
  \eqqcolon\epsilon_m.
\end{align}
Here $\epsilon_m\to0$, since $\sum_{k\in I}|a_k^{(m)}|^2$ is bounded.

For the index in $H$, let $\nu=\mu_{\ell+1}(A)>\mu$. Since $\lambda_{\sigma_\ell+1}(A_m)\to\nu$ and $g(t)\geq t-\mu$ for $t\geq\mu$, for all sufficiently large $m$ and all $k\in H$, there holds
\begin{align}
  g(\lambda_k(A_m))
  \geq\lambda_{\sigma_\ell+1}(A_m)-\mu
  \geq\frac{\nu-\mu}{2}.
\end{align}
For the index in $J$, $g(t)\geq|t|$ holds when $t<0$. Since $\lambda_{-1}(A_m)\to\lambda_{-1}(A)<0$, similarly, for all sufficiently large $m$ and all $k\in J$,
\begin{align}
  g(\lambda_k(A_m))
  \geq|\lambda_{-1}(A_m)|
  \geq\frac12|\lambda_{-1}(A)|.
\end{align}
Thus, letting
\begin{align}
  c\coloneqq\frac12\min\{\nu-\mu,|\lambda_{-1}(A)|\}>0,
\end{align}
there exists $m_0$ such that $g(\lambda_k(A_m))\geq c$ for every $m\geq m_0$ and every $k\in H\cup J$.

Now, setting $S_m\coloneqq\sum_{k\in H\cup J}|a_k^{(m)}|^2$ and combining the estimates above yields, for any $m\geq m_0$,
\begin{align}
  F_m\geq-\epsilon_m+cS_m.
\end{align}
Since $F_m\to0$ and $\epsilon_m\to0$, it follows that
\begin{align}
  0\leq cS_m\leq F_m+\epsilon_m\longrightarrow0.
\end{align}
Hence $S_m\to0$, which proves the claim.

    It follows that $\eta^{(m)}-P_\ell(A_m)\eta^{(m)}$ tends to zero in weighted norm. These spinors are uniformly bounded in $W^{1,p}$: the projection is a finite sum with bounded coefficients. Any uniform subsequential limit has zero $A$-norm and is therefore zero. Compactness now gives convergence to zero in $C^\alpha$ and $H^s$, $0<s<1$. Since $\eta^{(m)}\to\psi$ and the finite-rank projectors have uniformly bounded coefficient functionals on $C^0$, we also obtain $P_\ell(A_m)\psi\to\psi$ in these spaces.

    Finally, to obtain convergence in operator norm on $H^s$, take a common subsequence on which a weighted orthonormal basis of the cluster converges. Its limits form an $A$-orthonormal basis of $H_\ell(A)$, since the dimension is $h_\ell$. For each basis vector the functional
    \[
       u\longmapsto\int_M\Abracket{A_m u,\varphi_j(A_m)}\dv_g
    \]
    converges in the dual norm of $H^s$: use uniform convergence of the eigenspinors, weak $L^p$ convergence of $A_m$, and the compact embedding $H^s\Subset L^{p/(p-1)}$. The output basis converges in $H^s$, so the finite-sum formula for $P_\ell$ gives operator-norm convergence. The subsequence principle gives convergence of the whole sequence. Formula~\eqref{eq:harmonic projection coefficients} gives the corresponding assertion for $P_0(A_m)$.

    For the negative part of the spectrum, the same argument applied to $-\D_g$ gives, for every $\ell<0$,
    \begin{align}
       P_\ell(A_m)\longrightarrow P_\ell(A),\quad\text{as }m\to\infty.
    \end{align}
    This proves the claim and hence the theorem.
    \end{proof}

\subsection{Failure at the critical exponent}
    \begin{eg}\label{eg:critical}
    Let $M=(\R/2\pi\bZ)^n$, $n\ge2$, with its flat metric and the spin structure which is antiperiodic in $x^1$ and periodic in the other directions. Then $\ker\D=0$ and $\lambda_1(\id)=1/2$.
    Choose $0\le\eta\in C_c^\infty(B_2(0))$ with $\eta=1$ on $B_1(0)$. For sufficiently small $\eps>0$, set in a flat coordinate ball
    \[
    A_\eps=a_\eps\id,\qquad
    a_\eps(x)=1+\frac C\eps\eta(x/\eps),
    \]
    extending $a_\eps$ by $1$ outside the ball. These weights are smooth and positive. Their difference from $1$ has constant $L^n$ norm $C\|\eta\|_{L^n}$, and for $F\in L^{n/(n-1)}$,
    \[
    \left|\int_M(a_\eps-1)F\dv_g\right|
    \le C\|\eta\|_{L^n}\|F\|_{L^{n/(n-1)}(B_{2\eps})}\to0.
    \]
    Hence $A_\eps\rightharpoonup\id$ in $L^n$. Also $a_\eps^{-1}\to1$ strongly in every finite $L^q$, because $0<a_\eps^{-1}\le1$ and the support of $a_\eps^{-1}-1$ shrinks.

    In the complex notation for the underlying real spinor bundle, choose a constant unit spinor $v$ with $i\gamma(e_1)v=v$. For a nonzero real $f\in C_c^\infty(B_1(0))$ put $\varphi(y)=f(y)e^{iy^1}v$. Taking real parts of the Hermitian products,
    \[
    \int_{\R^n}\Abracket{\D\varphi,\varphi}=\int_{\R^n}f^2>0,
    \qquad \int_{\R^n}|\D\varphi|^2=\int_{\R^n}(|df|^2+f^2).
    \]
    For $\psi_\eps(x)=\varphi(x/\eps)$, its support lies where $a_\eps=1+C/\eps$, and scaling gives
    \[
    \int_M\Abracket{\D\psi_\eps,\psi_\eps}\dv_g
       =\eps^{n-1}\int_{\R^n}f^2,
    \qquad
    \int_Ma_\eps^{-1}|\D\psi_\eps|^2\dv_g
       =\frac{\eps^{n-1}}{C+\eps}\int_{\R^n}|\D\varphi|^2.
    \]
    The smooth-weight $k=1$ min--max formula yields
    \[
    \lambda_1(A_\eps)\le
    \frac{\int_{\R^n}|\D\varphi|^2}{(C+\eps)\int_{\R^n}f^2}.
    \]
    Choosing $C>2\int|\D\varphi|^2/\int f^2$ gives
    $\limsup_{\eps\downarrow0}\lambda_1(A_\eps)<1/2=\lambda_1(\id)$.
    Thus weak $L^n$ convergence does not guarantee eigenvalue continuity.
    \end{eg}

\subsection{Failure of strong first-order Sobolev convergence}
    \begin{eg}\label{eg:no H1}
    On the same flat spin torus, let $a_m=2+\sin(mx^1)$ and $A_m=a_m\id$. Then $A_m\rightharpoonup2\id$ in every finite $L^p$, and $1\le a_m\le3$. With $v$ as above and $c=(2\vol(M))^{-1/2}$, define
    \[
    \psi_m=c\exp\left(\frac{ix^1}{2}-\frac{i\cos(mx^1)}{4m}\right)v,
    \qquad \psi=c e^{ix^1/2}v.
    \]
    These spinors satisfy the required antiperiodicity, $\int_Ma_m|\psi_m|^2\dv_g=1$, and
    \[
    \D\psi_m=\tfrac14a_m\psi_m,\qquad
    \D\psi=\tfrac12\psi,\qquad \psi_m\to\psi\quad\text{uniformly}.
    \]
    In fact $\lambda_1(A_m)=1/4$. In the zero transverse Fourier mode the positive eigenvalues are $(2j+1)/4$, $j\ge0$, as follows by integrating the one-dimensional equation over one period. In a nonzero transverse Fourier mode, $\|\D u\|_{L^2}\ge\|u\|_{L^2}$ and $\|a_m u\|_{L^2}\le3\|u\|_{L^2}$, so every eigenvalue has absolute value at least $1/3$.
    However,
    \[
    \D(\psi_m-\psi)=\tfrac14\sin(mx^1)\psi+
                    \tfrac14a_m(\psi_m-\psi),
    \qquad
    \|\D(\psi_m-\psi)\|_{L^2}^2\to\tfrac1{32}\|\psi\|_{L^2}^2>0.
    \]
    There is no strong $H^1$ convergence. The $H^1$ distance to the limiting eigenspace also fails to tend to zero: any approximating sequence in that fixed finite-dimensional eigenspace would converge to $\psi$ in $L^2$, hence in $H^1$, giving a contradiction. For the first cluster, $P_1(A_m)\psi\to\psi$ uniformly and satisfies the same eigenvalue equation with eigenvalue $1/4$; the displayed calculation then also excludes convergence of these projections in $\mathcal B(H^1)$. This explains the strict range $s<1$ in Theorem~\ref{thm:continuity-closed}.
    \end{eg}

    \begin{rmk}\label{rmk:inverse H1}
    If, in addition, $A_m^{-1}\rightharpoonup A^{-1}$ in $L^p$, strong $H^1$ convergence is recovered. Indeed,
    \begin{align*}
    &\int_M\left|A_m^{-1/2}(A_m-A)A^{-1/2}\right|_{\mathrm{HS}}^2\dv_g\\
    &\qquad=\int_M\tr\bigl(A_mA^{-1}+A_m^{-1}A-2\id\bigr)\dv_g\longrightarrow0.
    \end{align*}
    Here the fixed factors $A,A^{-1}$ belong to $L^{p/(p-1)}$, since $p>2$. H\"older's inequality gives $A_m\to A$ in $L^1$, and interpolation with the uniform $L^p$ bound gives convergence in $L^2$. Along a uniformly convergent eigenspinor subsequence,
    \[
       \D_g(\varphi_m-\varphi)=\lambda_m A_m\varphi_m-\lambda A\varphi\longrightarrow0\quad\text{in }L^2.
    \]
    The elliptic estimate gives strong $H^1$ convergence, also when $\ker\D_g\ne0$. The same argument works whenever $A_m\to A$ strongly in $L^2$. The finite-dimensional projector argument then gives convergence in $B(H^1)$ as well. These conclusions also hold for the fixed chiral boundary condition.
    \end{rmk}

\section{Continuity of eigenvalues on manifolds with boundary}\label{sect:boundary}

In this section, we discuss the continuity of weighted eigenvalues under the chiral (CHI) boundary condition and aim to prove Theorem~\ref{thm:continuity-bdd}.
We begin by recalling basic facts concerning spin geometry on manifolds with boundary; for detailed discussions, we refer to~\cite{Ginoux2009Dirac}.

Let $(M, g)$ be an $n$-dimensional compact spin Riemannian manifold with nonempty boundary $\partial M$, equipped with a spin structure and associated spinor bundle $(\Sigma_g M, g^s)$, where $g^s$ denotes the induced metric on the spinor bundle $\Sigma_g M$. Assume that every connected component of $M$ meets the boundary. On the boundary $\partial M$, the restricted spinor bundle $\mathbf{S} = \Sigma_g M|_{\partial M}$ inherits a metric induced naturally by $g^s$, which we denote by $g^{\mathbf{S}}$. The boundary also carries the induced metric, orientation, and spin structure. The restricted spinor bundle $\mathbf{S}=\Sigma_g M|_{\partial M}$ is identified as follows:
$$
\mathbf{S}\cong
\begin{cases}
\Sigma \partial M & \text{if $n-1$ is even},\\[0.5em]
\Sigma \partial M\oplus \Sigma \partial M & \text{if $n-1$ is odd}.
\end{cases}
$$
In particular, for even $n$, this splitting corresponds to the decomposition
$$
\mathbf{S}=\Sigma^{+} M|_{\partial M}\oplus\Sigma^{-} M|_{\partial M}.
$$

On the boundary spinor bundle $\mathbf{S}$, the induced Clifford multiplication $\gamma^{\mathbf{S}}$ and spinorial connection $\snabla^{\mathbf{S}}$ are defined by
$$
\gamma^{\mathbf{S}}(X)\psi=\gamma(X)\gamma(\mathbf{n})\psi,\quad\snabla_X^{\mathbf{S}}\psi=\snabla_X\psi-\frac{1}{2}\gamma(\II(X))\gamma(\mathbf{n})\psi,
$$
where $\mathbf{n}$ denotes the inward unit normal vector on $\partial M$, and $\II(X)=-\nabla_X\mathbf n$ is the shape operator. These structures satisfy
$$
\snabla_X^{\mathbf{S}}(\gamma(\mathbf{n})\psi)=\gamma(\mathbf{n})\snabla_X^{\mathbf{S}}\psi.
$$

We assume the spinor bundle admits a global chirality operator $G$ with the properties below. The chiral boundary condition is defined by the projection operator
$$
\CHI^+=\frac{1}{2}(\mathrm{Id}-\gamma(\mathbf{n})G),
$$
where $G$ is an endomorphism-field of $\Sigma_g M$ (whose restriction on $\p M$ is also denoted by $G$) that is involutive, unitary, parallel, and anticommutes with Clifford multiplication, satisfying
$$
G^2=\mathrm{Id},\quad \langle G\psi,G\varphi\rangle=\langle\psi,\varphi\rangle,\quad\snabla_X(G\psi)=G\snabla_X\psi,\quad\gamma(X)G=-G\gamma(X),
$$
for all $X\in \Gamma(TM)$ and $\psi,\varphi\in\Gamma(\Sigma_g M)$.

It is known (see \cite{Ginoux2009Dirac}) that the Dirac operator $\D_g$ with chiral boundary condition $(\D_g,\CHI^+)$ has a real, discrete, and unbounded spectrum. We now study the following weighted eigenvalue problem:
\begin{equation}
\left\{
\begin{aligned}
\D_g\varphi^{\mathrm{CHI}}&=\lambda^{\mathrm{CHI}}A\varphi^{\mathrm{CHI}}&&\text{in }M,\\[0.5em]
\CHI^+\varphi^{\mathrm{CHI}}&=0&&\text{on }\partial M,
\end{aligned}
\right.
\label{eq:bddMAIN}
\end{equation}
where $A\in L^p(M,\mathrm{End}(\Sigma_g M))$, $p>n$, is symmetric and positive definite almost everywhere. The weighted realization is obtained as in Section~\ref{sect:minmax}, using the chiral elliptic boundary estimate.

We denote the weighted $\rm{CHI}$ eigenvalues of \eqref{eq:bddMAIN} by $\lambda_k^{\mathrm{CHI}}(A)$, with corresponding eigenspinors $\varphi_k^{\mathrm{CHI}}(A)$ satisfying
\begin{align}
A^{-1}\D_g\varphi_k^{\mathrm{CHI}}(A)=\lambda_k^{\mathrm{CHI}}(A)\varphi_k^{\mathrm{CHI}}(A),\quad k\in\mathbb{Z}_*.
\end{align}
These eigenvalues are real, discrete, and unbounded, indexed as follows:
\begin{align}
-\infty\leftarrow\cdots\leq\lambda_{-k}^{\mathrm{CHI}}(A)\leq\cdots\leq\lambda_{-1}^{\mathrm{CHI}}(A)<0<\lambda_1^{\mathrm{CHI}}(A)\leq\cdots\leq\lambda_k^{\mathrm{CHI}}(A)\leq\cdots\rightarrow+\infty.
\end{align}
Let $V_k^{\mathrm{CHI}}(A)\subset H^{1}(\Sigma_g M)$ be the finite-dimensional eigenspace associated with $\lambda_k^{\mathrm{CHI}}(A)$.
Additionally, we note that $(\D_g,\CHI^+)$ has no kernel~\cite[Corollary 3.7]{Chen2019estimates}, i.e.
\begin{align}
\ker(A^{-1}\D_g,\CHI^+) = \ker(\D_g,\CHI^+) = \{0\},
\end{align}
this kernel is independent of $A$.

We write the spectrum into its distinct weighted $\mathrm{CHI}$ eigenvalues. For each $\ell\in \mathbb{Z}_*$, denote by $\mu^{\mathrm{CHI}}_\ell(A)$ the $\ell$-th distinct weighted $\mathrm{CHI}$ eigenvalue, and set the associated eigenspace as
$H^{\mathrm{CHI}}_\ell(A) =\Eigen\big((A^{-1}\D_g,\CHI^+), \mu^{\mathrm{CHI}}_\ell(A) \big)$, whose dimension is $h_\ell =\dim H^{\mathrm{CHI}}_\ell(A)$. Let $P^{\mathrm{CHI}}_\ell(A)$ be the projector onto $H^{\mathrm{CHI}}_\ell(A)$. For $\ell > 0$, we set $\sigma_\ell\coloneqq\sum_{p=1}^\ell h_p$.
The eigenspaces can be conveniently labeled as follows:
\begin{align}
H^{\mathrm{CHI}}_1(A)&=\Span\big\{\varphi^{\mathrm{CHI}}_j(A)\mid 1\le j\le h_1\big\},\\
H^{\mathrm{CHI}}_\ell(A)&=\Span\big\{\varphi^{\mathrm{CHI}}_j(A)\mid \sigma_{\ell-1}< j\le \sigma_\ell\big\},\qquad \ell>1,
\end{align}
 and for~$\ell<0$, the eigenspaces and projectors are defined by a similar method. Moreover, for each $A_m$ the eigenspaces $H^{\mathrm{CHI}}_\ell(A_m)$ are still organized not with respect to the eigenvalues of $(A_m^{-1}\D_g,\CHI^+)$ but with respect to those of $(A^{-1}\D_g,\CHI^+)$, namely
            \begin{align}
                H^{\mathrm{CHI}}_1(A_m)=&\Span\braces{\varphi^{\mathrm{CHI}}_j(A_m)\mid 0<j\leq h_1}, \\
                H^{\mathrm{CHI}}_\ell(A_m)=&\Span\braces{\varphi^{\mathrm{CHI}}_j(A_m)\mid \sigma_{\ell-1} < j \leq \sigma_\ell   },\quad \ell>1,
            \end{align}
            and~$H^{\mathrm{CHI}}_\ell(A_m)$ for~$\ell<0$ are similarly defined, with the projector onto~$H^{\mathrm{CHI}}_\ell(A_m)$ denoted by~$P^{\mathrm{CHI}}_\ell(A_m)$ for each~$\ell\neq 0$.

We adopt the normalization conditions:
\begin{align}
\int_M\langle A\varphi_j^{\mathrm{CHI}}(A),\varphi_k^{\mathrm{CHI}}(A)\rangle_{g^s}\,dv_g&=\delta_{jk},\quad\forall j,k\in\mathbb{Z}_*.
\end{align}

Analogous to the case of closed manifolds, the continuity analysis for weighted \(\mathrm{CHI}\) eigenvalues and associated eigenprojections on compact manifolds with boundary will be carried out in several systematic steps.

\

\noindent{\bf Step 1}: We still use the min-max characterizations of weighted eigenvalues. Similar to the discussion on closed manifolds, we can define a Grassmannian manifold
$$Gr_{k-1}^*(A) :=\{V\subseteq H^{1}_{\mathrm{CHI}}(\Sigma_gM)| V\perp_A \ker(\D_g,\CHI^+), \dim V = k-1\}, $$
Here $H^1_{\mathrm{CHI}}=\{\psi\in H^1:\CHI^+\psi|_{\partial M}=0\}$, and every trial spinor in a reciprocal quotient is required to have finite reciprocal energy. The $A$-orthogonality $V \perp_A \ker(\D_g,\CHI^+)$ means~$\int_M \langle A\psi, \theta \rangle \, \mathrm{d}v_g = 0$, for any~$\psi \in V$ and~$\theta \in \ker(\D_g,\CHI^+)$.

The $k$-th weighted $\mathrm{CHI}$ positive eigenvalue is characterized by the following proposition.
 \begin{prop}\label{prop:minmax bdd}
        The~$k$-th positive weighted $\mathrm{CHI}$ eigenvalue~$\lambda_k^{\mathrm{CHI}}(A)$ is characterized by
\begin{align}\label{eq:minmax positive bdd}
\frac{1}{\lambda_k^{\mathrm{CHI}}(A)}
=& \inf_{V\in Gr^*_{k-1}(A)} \; \sup_{0\neq \psi\perp_A V}\braces{\frac{\int_M \Abracket{\psi,\D_g\psi}\dv_g }{\int_M \Abracket{A^{-1}\D_g\psi, \D_g\psi}\dv_g} }.
\end{align}
\end{prop}
\begin{proof}
 Let $\nu^{\mathrm{CHI}}_k$ denote the right-hand side. We show $\frac{1}{\lambda^{\mathrm{CHI}}_k(A)} = \nu^{\mathrm{CHI}}_k$.

For $V \in \mathrm{Gr}_{k-1}^*(A)$, the $k$-dimensional space spanned by $\{\varphi^{\mathrm{CHI}}_i(A)\}_{i=1}^k$ contains a spinor $\psi^{\mathrm{CHI}}_k$ with
\begin{equation}
    \int_M \langle A\psi^{\mathrm{CHI}}_k, \varphi \rangle_{g^s} \dv_g = 0 \quad \forall \varphi \in V.
\end{equation}
We assume $\psi^{\mathrm{CHI}}_k=\sum_{j=1}^k b_j\, \varphi^{\mathrm{CHI}}_j(A)$, and directly compute that
        \begin{align}
            \sup_{0\neq \psi\perp_A V}\braces{\frac{\int_M \Abracket{\psi,\D_g\psi}\dv_g }{\int_M \Abracket{A^{-1}\D_g\psi, \D_g\psi}\dv_g} }
            \geq &\frac{\int_M \Abracket{\psi^{\mathrm{CHI}}_k,\D_g\psi^{\mathrm{CHI}}_k}\dv_g }{\int_M \Abracket{A^{-1}\D_g\psi^{\mathrm{CHI}}_k, \D_g\psi^{\mathrm{CHI}}_k}\dv_g}
            =\frac{\sum_{i=1}^k \lambda^{\mathrm{CHI}}_i(A)\, b_i^2}{\sum_{i=1}^k \lambda^{\mathrm{CHI}}_i(A)^2 b_i^2} \\
            \geq& \frac{1}{\lambda^{\mathrm{CHI}}_k(A)}.
        \end{align}
Hence~$\nu^{\mathrm{CHI}}_k\geq \frac{1}{\lambda^{\mathrm{CHI}}_k(A)}$.

By taking~$V_*(A)=\Span\braces{\varphi^{\mathrm{CHI}}_1(A),\cdots,\varphi^{\mathrm{CHI}}_{k-1}(A)}$ and~$\psi=\varphi_k^{\mathrm{CHI}}(A)\perp_A V_*(A)$, we have
        $$
            \frac{\int_M \Abracket{\psi,\D_g\psi}\dv_g }{\int_M \Abracket{A^{-1}\D_g\psi, \D_g\psi}\dv_g} = \frac{1}{\lambda^{\mathrm{CHI}}_k(A)}.
        $$
For any other admissible $\psi\perp_A V_*(A)$, its weighted expansion has only indices $j\ge k$ or $j<0$, and
\[
\int_M\Abracket{\psi,\D_g\psi}\dv_g
=\sum_j\lambda_j^{\mathrm{CHI}}(A)b_j^2
\le\frac1{\lambda_k^{\mathrm{CHI}}(A)}\sum_j\lambda_j^{\mathrm{CHI}}(A)^2b_j^2.
\]
Thus the supremum of the quotient equals $1/\lambda_k^{\mathrm{CHI}}(A)$, and $\nu_k^{\mathrm{CHI}}\le1/\lambda_k^{\mathrm{CHI}}(A)$.
\end{proof}
Moreover, entirely similarly to the discussion on closed manifolds, we quickly obtain the following fact:
    \begin{rmk}
         The negative weighted $\mathrm{CHI}$ eigenvalues can be defined as:
        \begin{align}\label{eq:minmax negative bdd}
        \frac{1}{\lambda^{\mathrm{CHI}}_{-k}(A)}
        =& \sup_{V \in Gr^*_{k-1}(A)} \; \inf_{0\neq \psi\perp_A V}\braces{\frac{\int_M \Abracket{\psi,\D_g\psi}\dv_g }{\int_M \Abracket{A^{-1}\D_g\psi, \D_g\psi}\dv_g} }.
        \end{align}
    \end{rmk}

   An analogous min-max characterization of the weighted $\mathrm{CHI}$ eigenvalues follows as:
    \begin{prop}\label{prop:another minmax bdd}
        The $k$-th positive weighted $\mathrm{CHI}$ eigenvalue~$\lambda_k(A)$ is characterized by
        \begin{align}
            \frac{1}{\lambda^{\mathrm{CHI}}_k(A)}
            = \sup_{\widetilde{W}\in Gr^*_k(A)} \; \inf_{0\neq\psi, \psi\in \widetilde{W}}
            \braces{\frac{\int_M \Abracket{\psi,\D_g\psi}\dv_g }{\int_M \Abracket{A^{-1}\D_g\psi, \D_g\psi}\dv_g} },
        \end{align}
        and the~$k$-th negative weighted $\mathrm{CHI}$ eigenvalue~$\lambda_{-k}(A)$ is characterized by
        \begin{align}
            \frac{1}{\lambda^{\mathrm{CHI}}_{-k}(A)}
            = \inf_{\widetilde{W}\in Gr^*_k(A)} \; \sup_{0\neq\psi, \psi\in \widetilde{W}}
            \braces{\frac{\int_M \Abracket{\psi,\D_g\psi}\dv_g }{\int_M \Abracket{A^{-1}\D_g\psi, \D_g\psi}\dv_g} }.
        \end{align}
    \end{prop}

    We also obtain a comparison principle for such weighted $\mathrm{CHI}$ eigenvalues of the Dirac operator by the above min-max characterization of weighted $\mathrm{CHI}$ eigenvalues.

    \begin{proof}[Proof of Theorem~\ref{thm:comparison} on Compact Manifolds with Boundary]

    This directly follows from Proposition~\ref{prop:another minmax bdd}. For the compact spin manifolds with boundary, we note that the chiral eigenvalue problem has no kernel, i.e. $\ker(\D_g,\CHI^+)=\{0\}$, thus the Grassmann manifold $Gr^*_k(A)=Gr_k$ reduces to the classical Grassmann manifold independent of $A$.

    Take $\widetilde W=\Span\{\varphi_1^{\mathrm{CHI}}(A_2),\ldots,\varphi_k^{\mathrm{CHI}}(A_2)\}$. Given $A_1^{-1}\leq A_2^{-1}$, we have the inequality for weighted Rayleigh quotients:
\begin{align}
\frac{\int_M \Abracket{\psi,\D_g\psi}\dv_g }{\int_M \Abracket{A_1^{-1}\D_g\psi, \D_g\psi}\dv_g}
\geq
\frac{\int_M \Abracket{\psi,\D_g\psi}\dv_g }{\int_M \Abracket{A_2^{-1}\D_g\psi, \D_g\psi}\dv_g}, \quad 0\ne\psi\in\widetilde W.
\end{align}
The numerator is positive on this chosen space. Taking the infimum over its nonzero elements and using the supremum over admissible $k$-planes in Proposition~\ref{prop:another minmax bdd}, we obtain
\[
\frac1{\lambda_k^{\mathrm{CHI}}(A_1)}\ge
\inf_{0\ne\psi\in\widetilde W}
\frac{\int_M\Abracket{\psi,\D_g\psi}\dv_g}{\int_M\Abracket{A_1^{-1}\D_g\psi,\D_g\psi}\dv_g}
\ge\frac1{\lambda_k^{\mathrm{CHI}}(A_2)}.
\]
Therefore,
\begin{align}
\lambda^{\mathrm{CHI}}_k(A_1)\leq \lambda^{\mathrm{CHI}}_k(A_2),
\end{align}
as desired.

Applying the same argument to $-\D_g$ shows that for $k<0$, we have:
\begin{align}
\lambda^{\mathrm{CHI}}_k(A_1)\geq \lambda^{\mathrm{CHI}}_k(A_2).
\end{align}

    \end{proof}

\noindent{\bf Step 2}: Next, we derive a priori estimates for eigenspinors with the chiral boundary condition in~\eqref{eq:bddMAIN}. Under the previous assumptions, we recall that all these eigenspinors $\varphi^{\mathrm{CHI}}$ are normalized,
$$\int_M \Abracket{A\varphi^{\mathrm{CHI}}(A),\varphi^{\mathrm{CHI}}(A)} \dv_g = \int_M \left|\sqrt{A}\,\varphi^{\mathrm{CHI}}(A)\right|^2 \dv_g = 1,$$
where $\sqrt{A}$ denotes the unique symmetric positive definite square root of~$A$.
Following an approach analogous to the closed case, we immediately obtain the following estimates:
\begin{prop}\label{prop:regularity bdd}
        Let $(M,g)$ be an $n$-dimensional compact spin Riemannian manifold with boundary and~$\varphi^{\mathrm{CHI}}\in H^{1}(\Sigma_g M)$ be a solution of~\eqref{eq:bddMAIN} where~$A\in L^p(M,\End(\Sigma_g M))$ for some~$p>n$.
        There exist~$\alpha=\alpha(n,p)\in (0,1)$ and~$C=C(n,p,g)$ such that
        \begin{align}
            \|\varphi^{\mathrm{CHI}}\|_{H^{1}(M)}\leq C(n,p,g) \parenthesis{1+|\lambda^{\mathrm{CHI}}|\|A\|_{L^p} }^{T_1}
            |\lambda^{\mathrm{CHI}}|\|A\|_{L^p}^{\frac{1}{2}}
        \end{align}
        and
        \begin{align}
            \|\varphi^{\mathrm{CHI}}\|_{C^\alpha(M)}\leq C(n,p,g) \parenthesis{1+|\lambda^{\mathrm{CHI}}|\|A\|_{L^p} }^{T_2}
            |\lambda^{\mathrm{CHI}}|\|A\|_{L^p}^{\frac{1}{2}},
        \end{align}
        where
        \begin{align}
            T_1=\left\lceil\frac{n}{2(p-n)}\right\rceil, & &
            T_2=  \left\lfloor\frac{n(p-1)}{2(p-n)}\right\rfloor.
        \end{align}
       Here $\lceil\cdot\rceil$ and $\lfloor\cdot\rfloor$ denote the ceiling and floor functions, respectively. We choose $\alpha$ strictly below the final Morrey exponent, as in Proposition~\ref{prop:regularity}.
    \end{prop}
\begin{proof}
The chiral realization has trivial kernel, so its elliptic estimate gives, with $q_1=2p/(p+1)$,
\[
\|\varphi^{\mathrm{CHI}}\|_{W^{1,q_1}}
\le C\|\D_g\varphi^{\mathrm{CHI}}\|_{L^{q_1}}
\le C|\lambda^{\mathrm{CHI}}|\|A\|_{L^p}^{1/2}.
\]
For $q_j<n$, put $1/r_j=1/q_j-1/n$ and $1/q_{j+1}=1/p+1/r_j$. The Sobolev embedding and the boundary elliptic estimate yield
\[
\|\varphi^{\mathrm{CHI}}\|_{W^{1,q_{j+1}}}
\le C(1+|\lambda^{\mathrm{CHI}}|\|A\|_{L^p})\|\varphi^{\mathrm{CHI}}\|_{W^{1,q_j}}.
\]
Thus $1/q_j=1/q_1-(j-1)(1/n-1/p)$ until an exponent reaches $n$. At equality use $W^{1,n}\hookrightarrow L^r$ with finite $r$ large enough. Taking $q_j\ge2$ or $q_j>n$ gives the two asserted estimates with the displayed $T_1,T_2$. A final application of the equation gives a $W^{1,p}$ bound.
\end{proof}

\noindent{\bf Step 3}: We now turn to proving that every positive weighted $\mathrm{CHI}$ eigenvalue sequence is bounded, by establishing the uniform upper bound
\begin{equation}\label{eq:limsup bdd}
    \limsup_{m\to+\infty} \lambda^{\mathrm{CHI}}_k(A_m) \leq \lambda^{\mathrm{CHI}}_k(A).
\end{equation}

We recall the hypothesis on the weights:
\begin{align}\label{hypo bdd} \tag{$H_\partial$}
    &A_m, A \text{ are symmetric positive-definite endomorphisms satisfying} \\
    &A_m \rightharpoonup A \text{ weakly in } L^p(M,\mathrm{End}(\Sigma_g M)) \text{ for some } 2 \le n<p<\infty.
\end{align}
The precise version of Theorem~\ref{thm:continuity-bdd} is stated as follows.

\begin{thm}\label{thm:continuity-bdd}
        Let~$(M^n,g)$ be a compact Riemannian spin manifold with boundary,~$(\Sigma_g M, g^s,\snabla,\gamma)$ be a spinor bundle over~$M$ and~$A_m,A$ be weights satisfying~\eqref{hypo bdd}.
        Then, for each~$k\in \mathbb{Z}\setminus\braces{0}\coloneqq \mathbb{Z}_*$,
        \begin{enumerate}
            \item[(i)] $\lambda_k^{\mathrm{CHI}}(A_m)\to \lambda^{\mathrm{CHI}}_k(A)$ as~$m\to+\infty$;
            \item[(ii)] ~$\varphi^{\mathrm{CHI}}_k(A_m)$ converges to~$\Eigen((A^{-1}\D_g, \CHI^+);\lambda_k^{\mathrm{CHI}}(A))$, more precisely,
            \begin{align}
                \lim_{m\to+\infty} \inf_{\phi\in\Eigen((A^{-1}\D_g, \CHI^+);\lambda_k^{\mathrm{CHI}}(A))} \|\varphi_k^{\mathrm{CHI}}(A_m)-\phi\|_{C^\alpha}=0,
            \end{align}
            where~$\alpha \in (0,1)$ is given in Proposition~\ref{prop:regularity bdd}, and, for every $0<s<1$,
            \begin{align}
                \lim_{m\to+\infty} \inf_{\phi\in\Eigen((A^{-1}\D_g, \CHI^+);\lambda_k^{\mathrm{CHI}}(A))} \|\varphi_k^{\mathrm{CHI}}(A_m)-\phi\|_{H^{s}}=0.
            \end{align}
            \item[(iii)] Let~$P^{\mathrm{CHI}}_\ell(A_m)$ be the projection operator onto~$H^{\mathrm{CHI}}_\ell(A_m)$, for $\ell\in \mathbb{Z}_*$,
            then
            \begin{align}
                \lim_{m\to+\infty} P^{\mathrm{CHI}}_\ell(A_m)= P^{\mathrm{CHI}}_\ell(A),
            \end{align}
            where convergence is in operator norm on $H^s$, for every $0<s<1$.
            \end{enumerate}
    \end{thm}

Following exactly the same approach as in \eqref{eq:minmax positive bdd}, we can define the positive weighted $\mathrm{CHI}$ eigenvalues for $A_m$ via the min-max principle:

\begin{align}\label{eq:minmax positive bdd Am}
\frac{1}{\lambda_k^{\mathrm{CHI}}(A_m)}
=& \inf_{V\in Gr^*_{k-1}(A_m)} \; \sup_{0\neq \psi\perp_{A_m} V}\braces{\frac{\int_M \Abracket{\psi,\D_g\psi}\dv_g }{\int_M \Abracket{A_m^{-1}\D_g\psi, \D_g\psi}\dv_g} }.
\end{align}
In particular, the first weighted positive $\mathrm{CHI}$ eigenvalue for $A_m$ can be defined as:
\begin{equation}\label{eq:minmax positive 1bdd Am}
\frac{1}{\lambda_1^{\mathrm{CHI}}(A_m)}
=\sup_{0\neq \psi}\braces{\frac{\int_M \Abracket{\psi,\D_g\psi}\dv_g }{\int_M \Abracket{A_m^{-1}\D_g\psi, \D_g\psi}\dv_g} }.
\end{equation}

\begin{prop}
Let $(M, g)$ be an $n$-dimensional compact spin Riemannian manifold with boundary. Assume that hypothesis~\eqref{hypo bdd} holds on $M$, then for each positive weighted $\mathrm{CHI}$ eigenvalue with respect to $A_m$ for the equation \eqref{eq:bddMAIN}

$$
 \limsup_{m\to+\infty} \lambda^{\mathrm{CHI}}_k(A_m) \leq \lambda^{\mathrm{CHI}}_k(A), \quad k \geq 1.
$$
\end{prop}

\begin{proof}
For $i=1,\ldots,k$, let $\eta_i^{(m)}$ solve
\[
\D_g\eta_i^{(m)}=\lambda_i^{\mathrm{CHI}}(A)A_m\varphi_i^{\mathrm{CHI}}(A),
\qquad\CHI^+\eta_i^{(m)}=0.
\]
The boundary elliptic estimate gives a uniform $W^{1,p}$ bound. Compactness, weak convergence of $A_m$, and uniqueness of the limiting boundary problem give $\eta_i^{(m)}\to\varphi_i^{\mathrm{CHI}}(A)$ uniformly and in $H^s$, $0<s<1$.

If $k=1$, we take $\varphi=\eta_1^{(m)}$ as a test spinor. By characterization \eqref{eq:minmax positive 1bdd Am},
\begin{equation}
\frac{1}{\lambda^{\mathrm{CHI}}_1(A_m)} \geq \frac{\int_{M}\langle\D_g \varphi,\varphi\rangle \dv_g}{\int_{M} \langle A_m^{-1}\D_g\varphi,\D_g\varphi\rangle \dv_g}.
\label{eq:1eigvalrhon}
\end{equation}
Thus
$$
\liminf _{m \rightarrow \infty} \frac{1}{\lambda^{\mathrm{CHI}}_1\left(A_m\right)} \geq \liminf _{m \rightarrow \infty}\frac{\int_{M}\langle\D_g \varphi,\varphi\rangle \dv_g}{\int_{M} \langle A_m^{-1}\D_g\varphi,\D_g\varphi\rangle \dv_g} = \frac{1}{\lambda^{\mathrm{CHI}}_1(A)}.
$$
Indeed, the denominator equals
\[
\left[\lambda_1^{\mathrm{CHI}}(A)\right]^2\int_M\langle A_m\varphi_1^{\mathrm{CHI}}(A),\varphi_1^{\mathrm{CHI}}(A)\rangle\dv_g
\longrightarrow\left[\lambda_1^{\mathrm{CHI}}(A)\right]^2,
\]
and the numerator tends to $\lambda_1^{\mathrm{CHI}}(A)$. Hence~$\limsup_{m\to+\infty} \lambda^{\mathrm{CHI}}_1(A_m) \leq \lambda^{\mathrm{CHI}}_1(A)$.

If $k \geq 2$, we take a test spinor as
$$
\varphi^{(m)} = \sum_{i = 1}^k c_i^{(m)} \eta_i^{(m)},
$$
where $\eta_i^{(m)}$ are defined above, and the coefficients $c_i^{(m)}$, $i = 1,\ldots,k$ are determined by the following conditions
\begin{equation}
\left\{
\begin{array}{rl}
\displaystyle
\int_{M} \langle \varphi^{(m)}, A_m\varphi_i^{\mathrm{CHI}}(A_m) \rangle \, \dv_g &= 0, \quad \forall i = 1, \ldots, k-1, \\[10pt]
\displaystyle
\int_{M} \langle A_m\varphi^{(m)}, \varphi^{(m)} \rangle \, \dv_g &= 1.
\end{array}
\right.
\end{equation}
The test spinors are linearly independent for large $m$, and the normalization gives $\sum_i[c_i^{(m)}]^2\to1$. Here $\varphi^{\mathrm{CHI}}_i(A_m)$, $i=1, \cdots, k-1$ are eigenspinors associated with $\lambda^{\mathrm{CHI}}_i\left(A_m\right), i=$ $1, \cdots, k-1$. By characterization \eqref{eq:minmax positive bdd Am} for $\lambda^{\mathrm{CHI}}_k\left(A_m\right)$, we have

\begin{equation}
\begin{aligned}
\frac{1}{\lambda^{\mathrm{CHI}}_k\left(A_m\right)}
& \geq \frac{\int_{M} \langle \D_g \varphi^{(m)}, \varphi^{(m)} \rangle \, \dv_g}{\int_{M} \langle A_m^{-1}\D_g\varphi^{(m)},\D_g\varphi^{(m)}\rangle \, \dv_g} \\
&= \frac{\sum_{j=1}^k \left[c_j^{(m)}\right]^2 \lambda^{\mathrm{CHI}}_j(A)+o(1) }{\int_{M} \langle A_m^{-1}\D_g\varphi^{(m)},\D_g\varphi^{(m)}\rangle \, \dv_g}
\end{aligned}
\label{eq:keigvalAmbdd}
\end{equation}
During the execution of the above computations, we note that for any $i, j$, the corresponding eigenspinors are~$A$-orthonormalized in the sense that
$$
\int_{M}  \langle A\varphi^{\mathrm{CHI}}_i, \varphi^{\mathrm{CHI}}_j\rangle \dv_g=\delta_{ij}.
$$
Moreover, by analyzing the denominator on the right side of inequality \eqref{eq:keigvalAmbdd}, we have
\begin{equation}
\begin{aligned}
\int_M&\langle A_m^{-1}\D_g\varphi^{(m)},\D_g\varphi^{(m)}\rangle\dv_g\\
=&\sum_{i,j=1}^k c_i^{(m)}\lambda_i^{\mathrm{CHI}}(A)c_j^{(m)}\lambda_j^{\mathrm{CHI}}(A)
\int_M\langle A_m\varphi_i^{\mathrm{CHI}}(A),\varphi_j^{\mathrm{CHI}}(A)\rangle\dv_g\\
=&\sum_{i=1}^k[c_i^{(m)}]^2[\lambda_i^{\mathrm{CHI}}(A)]^2\\
&+\sum_{i,j=1}^k c_i^{(m)}\lambda_i^{\mathrm{CHI}}(A)c_j^{(m)}\lambda_j^{\mathrm{CHI}}(A)
\int_M\langle(A_m-A)\varphi_i^{\mathrm{CHI}}(A),\varphi_j^{\mathrm{CHI}}(A)\rangle\dv_g\\
\le&\left(1+kC_k^{(m)}\right)\sum_{i=1}^k[c_i^{(m)}]^2[\lambda_i^{\mathrm{CHI}}(A)]^2,
\end{aligned}
\label{eq:fenmuguji}
\end{equation}
where
\[
C_k^{(m)}:=\max_{1\le i,j\le k}\left|\int_M\langle(A_m-A)\varphi_i^{\mathrm{CHI}}(A),\varphi_j^{\mathrm{CHI}}(A)\rangle\dv_g\right|\longrightarrow0.
\]
The last inequality follows from Cauchy--Schwarz applied to $\sum_i|c_i^{(m)}|\lambda_i^{\mathrm{CHI}}(A)$. Similarly, the numerator differs from $\sum_i[c_i^{(m)}]^2\lambda_i^{\mathrm{CHI}}(A)$ by $o(1)$, using uniform convergence of the $\eta_i^{(m)}$. The errors are uniform for the bounded coefficient vectors.
Consequently, we obtain
\begin{align}
\frac{1}{\lambda^{\mathrm{CHI}}_k\left(A_m\right)} \geq  \frac{\sum\limits_{j=1}^k \left[c_j^{(m)}\right]^2 \lambda^{\mathrm{CHI}}_j(A)+o(1)}{\left(1 + kC_k^{(m)}\right)\sum\limits_{i=1}^k\left(c_i^{(m)}\right)^2[\lambda^{\mathrm{CHI}}_i(A)]^2}\geq\frac{1}{\left(1 + kC_k^{(m)}\right)\lambda^{\mathrm{CHI}}_k(A)}-o(1).
\end{align}
Because of~\eqref{hypo bdd}, we have~$C_k^{(m)} \rightarrow 0$ as $m \rightarrow \infty$.
Hence
\begin{align}
   \liminf _{m \rightarrow+\infty}\frac{1}{\lambda^{\mathrm{CHI}}_k\left(A_m\right)}\geq \liminf _{m \rightarrow+\infty}\frac{1}{\left(1 + kC_k^{(m)}\right)\lambda^{\mathrm{CHI}}_k(A)} = \frac{1}{\lambda^{\mathrm{CHI}}_k(A)}
\end{align}
Therefore, for each $ k \geq 1$,
\begin{align}
    \limsup_{m\to+\infty} \lambda^{\mathrm{CHI}}_k(A_m) \leq \lambda^{\mathrm{CHI}}_k(A).
\end{align}

\end{proof}

\noindent{\bf Step 4}: For each given~$k\geq 1$, we shall show that~$\{\varphi^{\mathrm{CHI}}_k(A_m)\}_{m=1}^{\infty}$ is a bounded sequence in~$H^1(M)$ and~$C^\alpha(M)$ for some~$\alpha:=\alpha(n,p)\in(0,1)$.

Since each $\varphi^{\mathrm{CHI}}_k(A_m)$ is a solution to the chiral boundary value problem
\begin{equation}
\begin{cases}
\D_g \varphi_k^{\mathrm{CHI}} = \lambda_k^{\mathrm{CHI}}(A_m) A_m \varphi_k^{\mathrm{CHI}}, & \text{in } M \\
\CHI^+\varphi_k^{\mathrm{CHI}} = 0, & \text{on } \partial M
\end{cases}
\end{equation}
and satisfies the normalization condition
$\int_M \langle A_m\varphi_k^{\mathrm{CHI}}(A_m), \varphi_k^{\mathrm{CHI}}(A_m) \rangle_g \mathrm{d}v_g = 1,$
according to Proposition~\ref{prop:regularity bdd}, we know that

 \begin{align}
            \|\varphi_k^{\mathrm{CHI}}(A_m)\|_{W^{1,2}(M)}\leq C(n,p,g) \parenthesis{1+\lambda_k^{\mathrm{CHI}}(A_m)\|A_m\|_{L^p} }^{T_1}
            \lambda_k^{\mathrm{CHI}}(A_m)\|A_m\|_{L^p}^{\frac{1}{2}}
        \end{align}
        and
        \begin{align}
            \|\varphi_k^{\mathrm{CHI}}(A_m)\|_{C^\alpha(M)}\leq C(n,p,g) \parenthesis{1+\lambda_k^{\mathrm{CHI}}(A_m)\|A_m\|_{L^p} }^{T_2}
            \lambda_k^{\mathrm{CHI}}(A_m)\|A_m\|_{L^p}^{\frac{1}{2}},
        \end{align}
Here, $T_1$ and $T_2$ are consistent with the results presented in Proposition~\ref{prop:regularity bdd}.

For the right-hand side of the inequalities, we have uniform bounds for~$\lambda_k^{\mathrm{CHI}}\left(A_m\right)$ in the previous step.
The~$L^p$ norms of~$A_m$ are uniformly bounded because of~\eqref{hypo bdd}.
Consequently, we conclude that $\left\{\varphi_k^{\mathrm{CHI}}\left(A_m\right)\right\}_{m=1}^{\infty}$ is bounded in $H^1(M)$ and in a H\"older space with exponent strictly greater than $\alpha$, and also in $W^{1,p}(M)$.

\

\noindent\textbf{Step 5}: Consider a convergent subsequence of eigenvalues $\{\lambda^{\mathrm{CHI}}_k(A_{m_\ell})\}_{\ell=1}^\infty$ from $\{\lambda^{\mathrm{CHI}}_k(A_i)\}_{i=1}^\infty$, with limit $\mu = \lim\limits_{\ell \to \infty} \lambda_k^{\mathrm{CHI}}(A_{m_\ell})$. Let the corresponding eigenspinors $\{\varphi^{\mathrm{CHI}}_k(A_{m_\ell})\}_{l=1}^\infty$ converge strongly in $C^\alpha(\Sigma_g M)\cap H^s(\Sigma_g M)$ for $0<s<1$, and weakly in $W^{1,p}$, to a limit $\Phi$.
Then $\mu\ge0$, and $(\mu,\Phi)$ satisfies the weighted eigenvalue equation
  \begin{equation}
    \begin{cases}
      \D_g \Phi = \mu A \Phi, & \text{in } M, \\
      \CHI^+ \Phi = 0, & \text{on } \partial M,
    \end{cases}
  \end{equation}
in the weak sense.
Indeed, for any smooth test spinor supported in the interior, we have
\begin{align}
  \int_M \langle \D_g \Phi, \eta \rangle \, \dv_g
  &= \lim_{\ell \to +\infty} \int_M \langle \D_g \varphi^{\mathrm{CHI}}_k(A_{m_\ell}), \eta \rangle \, \mathrm{d}v_g \\
  &= \lim_{\ell \to +\infty} \int_M \lambda_k^{\mathrm{CHI}}(A_{m_\ell}) \langle A_{m_\ell} \varphi^{\mathrm{CHI}}_k(A_{m_\ell}), \eta \rangle \, \dv_g \\
  &= \mu\int_M  \langle A \Phi, \eta \rangle \, \dv_g.
\end{align}
Thus, $\Phi \in H^1 \cap C^\alpha$ is a weak solution of the boundary problem above and satisfies the normalization condition
\[
  \int_M \langle A \Phi, \Phi \rangle \, \mathrm{d}v_g
  = \lim_{\ell \to \infty} \int_M \langle A_{m_\ell} \varphi^{\mathrm{CHI}}_k(A_{m_\ell}), \varphi^{\mathrm{CHI}}_k(A_{m_\ell}) \rangle \, \mathrm{d}v_g
  = 1.
\]
The trace map is continuous on $W^{1,p}$, so the limit also satisfies the chiral boundary condition
$$\CHI^+ \Phi = 0, \quad \text{on } \partial M.$$
Hence, $\Phi$ is an eigenspinor of $\D_g$ with weight $A$ and eigenvalue $\mu$.
The normalization excludes $\Phi=0$. Since the chiral realization has trivial kernel, $\mu=0$ is impossible, so $\mu>0$.
For $k=1$, the upper bound from Step 3 gives $0<\mu\le\lambda_1^{\mathrm{CHI}}(A)$, and hence equality.
Inductively, pass to a common further subsequence for the first $k$ normalized eigenspinors. Weak convergence of the weights and uniform convergence of these spinors show that their limits are $A$-orthonormal. If $\mu<\lambda_k^{\mathrm{CHI}}(A)$, all $k$ limits would lie in the positive spectral subspace below $\lambda_k^{\mathrm{CHI}}(A)$, which has dimension at most $k-1$. This contradiction, together with Step 3, gives $\mu=\lambda_k^{\mathrm{CHI}}(A)$. Every subsequence has a further subsequence with this limit, so the full eigenvalue sequence converges. The same compactness argument proves convergence of the eigenspinors to the limiting eigenspace in $C^\alpha$ and $H^s$, $0<s<1$.

Hence, for each $k\geq 1$,
\begin{align}
    \lim_{m \to +\infty} \lambda^{\rm CHI}_k(A_m) = \lambda_k^{\rm CHI}(A).
\end{align}
For completeness, let us note that the case of negative eigenvalues can be handled analogously by repeating the arguments of Steps~2--5. In particular, in Step~3, the min--max characterization of negative eigenvalues \eqref{eq:minmax negative bdd} leads to a reversed inequality, namely
\begin{align}
       \liminf_{m \to \infty} \lambda_k^{\mathrm{CHI}}(A_m) \;\geq\; \lambda_k^{\mathrm{CHI}}(A),
   \qquad \forall\, k<0,
\end{align}
which concludes the proof for all $k \in \mathbb{Z}_*$.

\

\noindent{\bf Step 6}:
We now prove the continuity of the projectors for the positive $A$-weighted $\mathrm{CHI}$ eigenvalues of the Dirac operator.

In \textbf{Step 5}, it was shown that $\lim_{m \to \infty}  P^{\mathrm{CHI}}_\ell(A)^\perp P^{\mathrm{CHI}}_\ell(A_m)  = 0$. Thus, it suffices to prove the converse:
\begin{align}
    \lim_{m \to \infty} P^{\mathrm{CHI}}_\ell(A_m)^\perp P^{\mathrm{CHI}}_\ell(A)  = 0.
\end{align}
Equivalently, for any $\psi \in H^{\mathrm{CHI}}_\ell(A)$ with $\int_M \langle A\psi, \psi \rangle \, \dv_g = 1$, we must show $P^{\mathrm{CHI}}_\ell(A_m)^\perp \psi \to 0$ in $L^2$ as $m \to \infty$. By induction, assume this holds for $1, \dots, \ell-1$; we now prove it for $\ell \geq 1$. The case $\ell=1$ follows similarly with simpler arguments.

For the test spinor, solve
\[
\D_g\eta^{(m)}=\mu_\ell^{\mathrm{CHI}}(A)A_m\psi,
\qquad\CHI^+\eta^{(m)}=0.
\]
The boundary elliptic estimate and uniqueness give uniform $W^{1,p}$ bounds and $\eta^{(m)}\to\psi$ in $C^\alpha$ and $H^s$, $0<s<1$. Its numerator tends to $\mu_\ell^{\mathrm{CHI}}(A)$, and its reciprocal energy equals
\[
\mu_\ell^{\mathrm{CHI}}(A)^2\int_M\langle A_m\psi,\psi\rangle\dv_g\longrightarrow\mu_\ell^{\mathrm{CHI}}(A)^2.
\]

For each $m\geq1$, expand $\eta^{(m)}$ as
\begin{align}
    \eta^{(m)} = \sum_{k \neq 0} a_k^{(m)} \varphi^{\mathrm{CHI}}_k(A_m).
\end{align}
Hypothesis \eqref{hypo bdd} implies
\begin{align}
    1 &= \int_M \langle A\psi, \psi \rangle \, \dv_g = \lim_{m \to \infty} \int_M \langle A_m\eta^{(m)},\eta^{(m)}\rangle \, \dv_g = \lim_{m \to \infty}  \sum_{k \neq 0} |a_k^{(m)}|^2  \\
    &= \lim_{m \to \infty} \left( \| P^{\mathrm{CHI}}_\ell(A_m)\eta^{(m)}\|_{L^2_{A_m}}^2 + \sum_{1\leq k \leq \sigma_{\ell-1}} |a_k^{(m)}|^2 + \sum_{k > \sigma_\ell } |a_k^{(m)}|^2 + \sum_{k < 0} |a_k^{(m)}|^2 \right)
\end{align}
The first term is the squared weighted norm of the cluster projection. For each lower positive index, Step 5 and uniform convergence of $\eta^{(m)}$ give
\[
a_k^{(m)}=\int_M\langle A_m\eta^{(m)},\varphi_k^{\mathrm{CHI}}(A_m)\rangle\dv_g\longrightarrow0,
\qquad1\le k\le\sigma_{\ell-1},
\]
because the limiting eigenspinor is $A$-orthogonal to $\psi$.
Thus,
\begin{equation}
    \lim_{m\to\infty} \sum_{1\leq k\leq \sigma_{\ell-1}} |a^{(m)}_k|^2 = 0.
\end{equation}
Define $\delta_m \coloneqq \sum\limits_{1\leq k\leq \sigma_{\ell-1}} a_k^{(m)}\varphi^{\mathrm{CHI}}_k(A_m)$, we consider the spinor
\begin{align}
    \phi_m &\coloneqq \eta^{(m)} - \delta_m \notag \\
    &= \sum_{\sigma_{\ell-1} < k \leq \sigma_{\ell}} a_k^{(m)}\varphi^{\mathrm{CHI}}_k(A_m)
    + \sum_{k > \sigma_{\ell}} a_k^{(m)}\varphi^{\mathrm{CHI}}_k(A_m)
    + \sum_{k<0} a_k^{(m)}\varphi^{\mathrm{CHI}}_k(A_m),
\end{align}
where $\delta_m\to0$ in $W^{1,p}$, by the vanishing coefficients and the uniform bounds from Step 4. Its reciprocal energy also tends to zero. We have:
\begin{gather}
    \lim_{m\to\infty} \int_M \langle A_m\phi_m, \phi_m \rangle \dv_g = 1, \\
    \phi_m \perp_{A_m} \Span\{\varphi^{\mathrm{CHI}}_1(A_m),\dots,\varphi^{\mathrm{CHI}}_{\sigma_{\ell-1}}(A_m)\}, \\
    \lim_{m\to\infty} \int_M \langle \D_g\phi_m, \phi_m \rangle \dv_g = \mu^{\mathrm{CHI}}_\ell(A), \\
    \lim_{m\to\infty} \int_M \langle A_m^{-1}\D_g\phi_m, \D_g\phi_m \rangle \dv_g = \left|\mu^{\mathrm{CHI}}_\ell(A)\right|^2.
\end{gather}
We claim
\begin{equation}
    \lim_{m\to\infty} \left( \sum_{k > \sigma_\ell} |a_k^{(m)}|^2 + \sum_{k<0} |a_k^{(m)}|^2 \right) = 0.
\end{equation}
Suppose for contradiction that
\begin{equation}
    r \coloneqq \limsup_{m\to\infty} \left( \sum_{k > \sigma_{\ell}} |a_k^{(m)}|^2 + \sum_{k<0} |a_k^{(m)}|^2 \right) > 0.
\end{equation}
Passing to a subsequence, assume the sum tends to $r>0$. Then:
\begin{align}
    \int_M \langle \D_g\phi_m, \phi_m \rangle \dv_g
    &= \sum_{\sigma_{\ell-1} < k \leq \sigma_{\ell}} \lambda^{\mathrm{CHI}}_k(A_m)|a_k^{(m)}|^2
    + \sum_{k > \sigma_{\ell}} \lambda^{\mathrm{CHI}}_k(A_m)|a_k^{(m)}|^2\\
    &\quad + \sum_{k<0} \lambda^{\mathrm{CHI}}_k(A_m)|a_k^{(m)}|^2, \\
    \int_M \langle A_m^{-1}\D_g\phi_m, \D_g\phi_m \rangle \dv_g
    &= \sum_{\sigma_{\ell-1} < k \leq \sigma_{\ell}} \lambda^{\mathrm{CHI}}_k(A_m)^2|a_k^{(m)}|^2
    + \sum_{k > \sigma_{\ell}} \lambda^{\mathrm{CHI}}_k(A_m)^2|a_k^{(m)}|^2\\
    &\quad + \sum_{k<0} \lambda^{\mathrm{CHI}}_k(A_m)^2|a_k^{(m)}|^2.
\end{align}
The convergence of the indexed eigenvalues gives a number $d>0$ such that
\[
\lambda_k^{\mathrm{CHI}}(A_m)\bigl(\lambda_k^{\mathrm{CHI}}(A_m)-\mu_\ell^{\mathrm{CHI}}(A)\bigr)\ge d
\quad\text{for }k>\sigma_\ell\text{ or }k<0,
\]
for all sufficiently large $m$. Indeed, $\lambda_{\sigma_\ell+1}^{\mathrm{CHI}}(A)>\mu_\ell^{\mathrm{CHI}}(A)>0$ and $\lambda_{-1}^{\mathrm{CHI}}(A)<0$. For the indices within the cluster the same expression tends to zero. Subtracting the preceding expansions therefore yields
\begin{align}
\int_M\langle A_m^{-1}\D_g\phi_m,\D_g\phi_m\rangle\dv_g
-\mu_\ell^{\mathrm{CHI}}(A)\int_M\langle\D_g\phi_m,\phi_m\rangle\dv_g\ge d\left(\sum_{k>\sigma_\ell}|a_k^{(m)}|^2+\sum_{k<0}|a_k^{(m)}|^2\right)-o(1).
\end{align}
The left-hand side tends to zero, whereas the right-hand side tends to $dr>0$, yielding a contradiction.

Thus $\eta^{(m)}-P_\ell^{\mathrm{CHI}}(A_m)\eta^{(m)}$ tends to zero in weighted norm. It is uniformly bounded in $W^{1,p}$, since the projection is a finite sum with bounded coefficients. Any uniform subsequential limit has zero $A$-norm, hence vanishes. Compactness gives convergence to zero in $C^\alpha$ and $H^s$, $0<s<1$. Since $\eta^{(m)}\to\psi$ and the finite-rank projectors have uniformly bounded coefficient functionals on $C^0$, we conclude $P_\ell^{\mathrm{CHI}}(A_m)\psi\to\psi$ in these spaces.

For operator-norm convergence, pass to a common subsequence on which the $h_\ell$ normalized eigenspinors in the cluster converge. Their limits form an $A$-orthonormal basis of $H_\ell^{\mathrm{CHI}}(A)$. In the formula
\[
P_\ell^{\mathrm{CHI}}(A_m)u
=\sum_{\sigma_{\ell-1}<j\le\sigma_\ell}
\left(\int_M\langle A_mu,\varphi_j^{\mathrm{CHI}}(A_m)\rangle\dv_g\right)\varphi_j^{\mathrm{CHI}}(A_m),
\]
the output vectors converge in $H^s$, and the coefficient functionals converge in its dual norm by weak $L^p$ convergence and $H^s\Subset L^{p/(p-1)}$. This proves operator-norm convergence on $H^s$; the subsequence principle gives convergence of the whole sequence.

By the same reasoning, for every~$\ell<0$, we have
\begin{align}
P^{\mathrm{CHI}}_\ell(A_m)\rightarrow P^{\mathrm{CHI}}_\ell(A) \quad \text{as}\quad m\rightarrow \infty.
\end{align}
This completes the proof.

\section{Application in geometric and variational problems}\label{sect:applications}

    In this section we consider some geometric variational problems, as applications to our results.
    We will be sketchy in these discussions, more details will be given in the analysis of the concrete problems.

%

\subsection{Super Liouville equations}
    In this subsection we consider an equation arising from the super Liouville type equations, see~\cite{Han2025superLiouville,Jevnikar2020existence,Jevnikar2021SuperLiouville,Jevnikar2021sinhGordon, Jost2007super} and the references therein.
    Let~$(M,g)$ be a Riemann surface, with or without boundary.
    Let~$u\in H^1(M)$ be a function and consider spinors~$\psi$ solving the equation
    \begin{align}\label{eq:sL}
        \D_g\psi= \lambda e^u \psi.
    \end{align}
    This is a weighted eigenvalue equation with~$A= e^{u}\id_{\Sigma_g M}$.
    Moreover, the Moser--Trudinger inequality tells that~$e^{\pm u}\in L^p$ for any~$p\in [1,+\infty)$, but~$e^{\pm u} \notin L^\infty(M)$ in general.

    Now let~$u_m\to u$ in~$H^1(M)$ and let~$\psi_k(e^{u_m})$ be a sequence of eigenspinors to the eigenvalue~$\lambda_k(e^{u_m})$ which are bounded in~$L^2$.
    Moser--Trudinger and uniform integrability give $e^{\pm u_m}\to e^{\pm u}$ in every finite $L^p$. Thus the $H^1$ conclusion of Remark~\ref{rmk:inverse H1} applies. Our results tell us that
    \begin{enumerate}
        \item the weighted eigenvalues converge~$\lambda_k(e^{u_m})\to \lambda_k(e^u)$;
        \item the weighted eigenspinors may not converge, but there is always a convergent subsequence;
        \item the projectors~$P_\ell(e^{u_m})$ converge to~$P_\ell(e^u)$ for each spectral point~$\mu_\ell$.
    \end{enumerate}
    This will be useful in the variational analysis of super Liouville equations on surfaces.

    Indeed, consider an eigenpair of the form
    \begin{align}
        \D_g \varphi_k(u) = \lambda_k(u) \, e^u \varphi_k(u).
    \end{align}
    If~$u$ is smooth (or at least~$C^2$), then the conformal metric~$g_u\coloneqq e^{2u}g$ induces a new spin structure as well as an associated spin Dirac bundle.
    There exists a unitary isomorphism~$\beta\colon \Sigma_g M \to \Sigma_{g_u}M$ which relates the two Dirac operators~$\D_g$ and~$\D_{g_u}$ via
    \begin{align}
        \D_{g_u}\parenthesis{ e^{-\frac{n-1}{2}u}\beta(\psi)}
        = e^{-\frac{n+1}{2}u}\beta(\D_g\psi), \qquad \forall \psi\in \Gamma(\Sigma_g M),
    \end{align}
    see~\cite{Ginoux2009Dirac,Hitchin1975Harmonic,LawsonMichelsohn1989spin} and~\cite{Jost2018symmetries}.
    In particular,
    \begin{align}
        \D_{g_u}\parenthesis{ e^{-u/2} \beta(\varphi_k(u))}
        = \lambda_k(u) \parenthesis{ e^{-u/2} \beta(\varphi_k(u)) }.
    \end{align}
    That is,~$(\lambda_k(u), e^{-u/2}\beta(\varphi_k(u)))$ is an eigenpair for the Dirac operator with respect to the conformal metric~$g_u$.
    For smooth~$u$ the eigenpairs are continuous in the conformal factor~$u$.
    For $u\in H^1(M)$, our results define the corresponding spectral data through the weighted equation~\eqref{eq:sL} and show their continuity in $u$.
     This allows us to consider such rough conformal factors and talk about the corresponding Dirac eigenvalues.

\subsection{Continuity of the weighted wave operator}

For the $H^1$ statements in this subsection, we assume that $A_m,A$ are smooth positive weights satisfying~\eqref{hypo} and $A_m\to A$ strongly in $L^2$. Remark~\ref{rmk:inverse H1} then gives $H^1$ convergence of the spectral subspaces and projectors. We use complexified spinors for the wave evolution. All kernel expansions below are understood distributionally.

Motivated by the spectral perspective developed in the works of G$\hat{a}$tel-Yafaev~\cite{Gatel2001Scattering} and Kungsman-Melgaard~\cite{Kungsman2017Wave}, we consider an $A$-weighted formula of the wave operator for Dirac-type operators on compact manifolds.

In the long-range setting, G$\hat{a}$tel and Yafaev introduced time-independent modified wave operators constructed via pseudo-differential operator techniques, revealing that the scattering behaviour of the Dirac equation is effectively captured through smooth spectral perturbations. Their analysis relies on the limiting absorption principle and radiation estimates, leading to precise asymptotics of wave evolution at large times. Moreover, in a related direction, Kungsman and Melgaard established a Poisson-type wave trace formula for self-adjoint Dirac operators perturbed by compactly supported potentials, where resonances appear as poles of the continued resolvent~\cite{Kungsman2017Wave}.
Their global trace formula decomposes the wave evolution into resonance contributions and discrete spectral terms, demonstrating that the wave operator encodes detailed spectral data.

These developments motivate consideration of the wave kernel $K(t,x,y)$ associated with the wave operator $e^{it\D}$, defined by
\begin{align}
  e^{it\D}\psi_0(x) = \int_M K(t,x,y) \psi_0(y) \, \dv_g(y), \quad \forall \psi_0 \in {H^1}(\Sigma_g M). 
\end{align}
In addition, the wave kernel $K(t,x,y)$ 
can be effectively expressed as a linear combination of oscillatory weights and spectral projections. In this framework, the asymptotic and analytic behaviour of the wave kernel becomes closely tied to the convergence properties of eigenvalues and spectral projections. Consequently, the $A$-weighted wave kernel approach offers a refined analytic structure under perturbation in the sense of the weak topology.

Let $(M, g)$ be a closed spin Riemannian manifold and its associated Dirac operator $\D_g$. Given the complete discrete spectral decomposition $\ker\D_g\oplus\{\lambda_k, \psi_k\}_{k\in \mathbb{Z}_*}$ of $\D_g$, where each eigenvalue $\lambda_k$ is listed according to its multiplicity and the sequence includes both positive and negative eigenvalues and $\psi_k$ denotes the eigenspinor corresponding to $\lambda_k$, the properties related to $\ker\D_g$ can be found in Section \ref{sect:minmax}. The integral kernel of the wave operator admits the representation:
\begin{align}
    K(t,x,y) = \sum_{k\in\mathbb{Z}_*} e^{it\lambda_k} \psi_k(x) \otimes \psi_k^*(y) + \sum_{k=1}^{h_0}\theta_k(x)\otimes \theta_k^*(y),
\end{align}
where $\psi_k^* \in \Sigma^*_g M$ denotes the dual spinor at point $y$, and $\psi_k(x) \otimes \psi_k^*(y)$ defines a rank-1 operator from $\Sigma_y M$ to $\Sigma_x M$.

We observe that expressions of the form
\begin{align}
    P_\ell(x,y) \coloneqq \sum_{k \in [\mu_\ell]} \psi_k(x) \otimes \psi_k^*(y),\qquad P_0(x,y) \coloneqq \sum_{k=1}^{h_0}\theta_k(x)\otimes \theta_k^*(y)
\end{align}
can be interpreted as the spectral projection operator $P_\ell(x,y)$ onto the eigenspace associated with the $\ell$-th distinct eigenvalue, where $[\mu_\ell]$ denotes the index set $\{k\in\mathbb{Z}_*\mid \lambda_k = \mu_\ell\}$. Thus,
this expression can equivalently be written using the spectral projection operators
\begin{align}
    K(t,x,y) = \sum_{\ell\in\mathbb{Z}_*} e^{it\mu_\ell} P_\ell(x,y) + P_0(x,y).
\end{align}
This formulation shows that the wave kernel is a linear combination of spectral projection kernels, modulated by the oscillatory weights.

Now we focus on a sequence of weighted Dirac-type operators $A_m^{-1}\D_g$ on closed manifolds, with eigenvalues $\lambda_k(A_m)$ repeated according to multiplicity. By~\eqref{eq:A-orthonormal}, choose a weighted orthonormal eigenbasis $\{\psi_k(A_m)\}_{k\in\mathbb Z_*}$ together with a weighted orthonormal harmonic basis $\{\theta_i(A_m)\}_{i=1}^{h_0}$. The corresponding wave kernel is
\begin{align}
K_m(t,x,y)
&=\sum_{k\in\mathbb Z_*}e^{it\lambda_k^{(m)}}\psi_k^{(m)}(x)\otimes\bigl(A_m(y)\psi_k^{(m)}(y)\bigr)^*\\
&\quad+\sum_{i=1}^{h_0}\theta_i(A_m)(x)\otimes\bigl(A_m(y)\theta_i(A_m)(y)\bigr)^*,
\end{align}
where $\lambda_k^{(m)}=\lambda_k(A_m)$ and $\psi_k^{(m)}=\psi_k(A_m)$.

Similarly to the discussion of classical spectral projections, we have
\begin{align}
P_\ell^{(m)}(x,y)
&=\sum_{k\in[\mu_\ell^{(m)}]}\psi_k^{(m)}(x)\otimes\bigl(A_m(y)\psi_k^{(m)}(y)\bigr)^*,\\
P_0^{(m)}(x,y)
&=\sum_{i=1}^{h_0}\theta_i(A_m)(x)\otimes\bigl(A_m(y)\theta_i(A_m)(y)\bigr)^*;
\end{align}
then the $A_m$-weighted wave operator can be written as
\begin{align}
K_m(t,x,y) = \sum_{\ell\in\mathbb{Z}_*} e^{it\mu_\ell^{(m)}} P_\ell^{(m)}(x,y) + P_0^{(m)}(x,y),
\end{align}
where $\mu^{(m)}_\ell$ denotes the $\ell$-th distinct eigenvalue of $A_m^{-1}\D_g$ and $[\mu^{(m)}_\ell]$ denotes $\{k\in\mathbb{Z}_*\mid \lambda^{(m)}_k = \mu^{(m)}_\ell\}$.

We assume that $A$ is the weak limit of the above $\{A_m\}_{m=1}^{\infty}$ in $L^p$, in a completely analogous discussion, we have an $A$-weighted wave kernel
\begin{align}
K_A(t,x,y) = \sum_{\ell\in\mathbb{Z}_*} e^{it\mu_\ell(A)} P_\ell^{A}(x,y)+P_0^A(x,y),
\end{align}
where $\mu_\ell(A)$ denotes the $\ell$-th distinct eigenvalue of $A^{-1}\D_g$ and $P_\ell^{A}$ is the spectral projection operator with respect to $\mu_\ell(A)$.
We emphasize that we do not claim pointwise convergence of the corresponding Schwartz kernels.

\

In light of Theorem~\ref{thm:continuity-closed}, we have established the convergence of both eigenvalues and the corresponding spectral projections for the $A_m$-weighted Dirac operators on a closed spin manifold. In particular, for each fixed $k \in \mathbb{Z}_*$, the eigenvalues $\lambda_k(A_m)$ converge to $\lambda_k(A)$, and the projectors onto the indexed spectral clusters converge in $B(H^1)$. The dimension of a cluster is fixed, but its eigenvalue may split into several distinct eigenvalues for $A_m$. Individual distinct eigenvalues and their multiplicities therefore cannot be identified across the sequence.

This kernel provides the integral representation of the wave operator
\begin{align}
e^{it(A^{-1}\D_g)}: H^1(\Sigma_g M) \rightarrow L^2(\Sigma_g M)
\end{align}
as:
\begin{align}
    \left(e^{it(A^{-1}\D_g)} \psi_0\right)(x) = \int_M K_A(t,x,y) \psi_0(y) \, \dv_g(y), \quad \forall \psi_0 \in \Gamma_{H^1}(\Sigma_g M),
\end{align}
by the spectral expansion above.

Let $\{\theta_i\}_{i=1}^{h_0}$ be a fixed $L^2$-orthonormal basis of $\ker\D_g$, independent of the weights. For weight $A$, the harmonic projection is $P_0^A\varphi=\sum_i b_i\theta_i$, where
\[
\sum_i b_i\int_M\langle A\theta_i,\theta_j\rangle\dv_g
=\int_M\langle A\varphi,\theta_j\rangle\dv_g,
\qquad j=1,\ldots,h_0.
\]
For $P_0^{A_m}$, replace $A$ by $A_m$ in this system. The matrices converge to a positive definite matrix. The right-hand side functionals converge in the dual norm of $H^1$, by weak $L^p$ convergence and $H^1\Subset L^{p/(p-1)}$. Consequently,
\[
P_0^{A_m}\longrightarrow P_0^A\quad\text{in }B(H^1).
\]

In the proof of weak operator convergence for the wave operator, one has to control truncation errors coming
from spectral tails. This requires estimates of the form $\|(I-Q_m^{(N)})u\|$, and hence some \emph{strong}
information on the spectral projections. A convenient observation is that, for orthogonal projections, weak operator
convergence automatically upgrades to strong convergence. We record this standard fact for later use.

\begin{lemma}\label{lem:WOT-to-SOT-proj}
Let $H$ be a Hilbert space and let $P_m,P\in B(H)$ be orthogonal projections, i.e.\ $P_m=P_m^* =P_m^2$ and
$P=P^* =P^2$. If $P_m\to P$ in the weak operator topology, then $P_m\to P$ in the strong operator topology; namely,
for every $u\in H$ one has $\|P_mu-Pu\|\to 0$.
\end{lemma}

\begin{proof}
Fix $u\in H$. Since $P_m\to P$ weakly, we have $P_mu\rightharpoonup Pu$. Moreover, because $P_m$ is an orthogonal
projection,
\[
\|P_mu\|^2=\langle P_mu,u\rangle.
\]
Passing to the limit in the right-hand side yields $\|P_mu\|^2\to \langle Pu,u\rangle=\|Pu\|^2$. Thus $P_mu$
converges weakly to $Pu$ and its norms converge to $\|Pu\|$, which implies $P_mu\to Pu$ strongly.
\end{proof}

\medskip
In particular, for projections orthogonal in one fixed Hilbert space, weak convergence suffices to obtain the strong convergence needed for spectral truncation. The weighted projections here are orthogonal for different inner products; we instead use their operator-norm convergence established above.

The dimensions of individual eigenspaces may vary along the sequence. We therefore use the consistency convention of Section~\ref{sect:Intro}, collecting the same indexed eigenvectors in each cluster. Write
\[
U_A(t)=e^{it(A^{-1}\D_g)},\qquad U_m(t)=e^{it(A_m^{-1}\D_g)}.
\]
Let $P_\ell^A=P_\ell(A)$ and $P_\ell^{A_m}=P_\ell(A_m)$ be the corresponding projectors, and let $P_0^A,P_0^{A_m}$ be the harmonic projectors.

\begin{lemma}\label{lem:QA-strong}
For every $u\in H^1(\Sigma_g M)$ one has
\begin{align}
Q_A^{(N)}u \longrightarrow u
\quad\text{in }H^1,
\qquad N\to\infty,
\end{align}
where
\begin{align}
Q_A^{(N)} := P_0^{A}+\sum_{\substack{\ell\in\mathbb Z_*\\|\ell|\le N}} P_\ell^{A}.
\end{align}
Equivalently, for every $\varepsilon>0$ there exists $N$ such that
\begin{align}
\|(I-Q_A^{(N)})u\|_{H^1}<\varepsilon.
\end{align}
\end{lemma}

\begin{proof}
Since $A^{-1}\D_g$ has compact resolvent, its eigenspinors form a complete orthonormal basis for the weighted $L^2$ inner product.
For the fixed smooth positive weight $A$, the $H^1$-norm is equivalent to the weighted graph norm of $A^{-1}\D_g$; hence there exists $C\ge 1$ such that for all
$u\in H^1$,
\begin{align}
C^{-1}\|u\|_{H^1}^2
\le
\|u\|_{L^2_A}^2+\|A^{-1}\D_gu\|_{L^2_A}^2
=
\|P_0^{A}u\|_{L^2_A}^2+\sum_{\ell\in\mathbb Z_*}(1+\mu_\ell(A)^2)\,\|P_\ell^{A}u\|_{L^2_A}^2
\le
C\|u\|_{H^1}^2.
\end{align}
Therefore,
\begin{align}
\|(I-Q_A^{(N)})u\|_{H^1}^2
\;\lesssim\;
\sum_{\substack{\ell\in\mathbb Z_*\\|\ell|>N}}(1+\mu_\ell(A)^2)\,\|P_\ell^{A}u\|_{L^2_A}^2,
\end{align}
and the right-hand side is the tail of a convergent series, hence it tends to $0$ as $N\to\infty$.
\end{proof}

\medskip

By Theorem~\ref{thm:continuity-closed} and Remark~\ref{rmk:inverse H1}, for every fixed $\ell\in\mathbb Z_*$ we have
\begin{align}
\lambda_k(A_m)\to\mu_\ell(A)\quad(k\in[\mu_\ell(A)]),
\qquad
P_\ell^{A_m}\to P_\ell^{A}\ \ \text{in }B(H^1),
\qquad
P_0^{A_m}\to P_0^{A}\ \text{on }H^1.
\end{align}
For $N\in\mathbb N$ we define 
\begin{align}
Q_m^{(N)} &:= P_0^{A_m}+\sum_{\substack{\ell\in\mathbb Z_*\\|\ell|\le N}} P_\ell^{A_m},\\
U_m^{(N)}(t)u &:= P_0^{A_m}u+\sum_{\substack{\ell\in\mathbb Z_*\\|\ell|\le N}}
\sum_{k\in[\mu_\ell(A)]}e^{it\lambda_k(A_m)}
\left(\int_M\langle A_mu,\psi_k(A_m)\rangle\dv_g\right)\psi_k(A_m),\\
U_A^{(N)}(t) &:= P_0^{A}+\sum_{\substack{\ell\in\mathbb Z_*\\|\ell|\le N}} e^{it\mu_\ell(A)}\,P_\ell^{A}.
\end{align}

\begin{prop}\label{prop:WOT-Awave}
Fix $t\in\mathbb R$ and assume $U_A(t)$ acts boundedly on $H^1(\Sigma_g M)$ and $\sup_m\|U_m(t)\|_{B(H^1)}<\infty$.
Then
\begin{align}
U_m(t)\longrightarrow U_A(t)
\quad\text{in the weak operator topology on }B(H^1),
\end{align}
i.e.\ for every $u,v\in H^1(\Sigma_g M)$,
\begin{align}
\big\langle U_m(t)u,\,v\big\rangle_{H^1}
\longrightarrow
\big\langle U_A(t)u,\,v\big\rangle_{H^1}.
\end{align}
\end{prop}

\begin{proof}
Fix $u,v\in H^1$ and $\varepsilon>0$.
By Lemma~\ref{lem:QA-strong}, choose $N$ such that
\begin{align}
\|(I-Q_A^{(N)})u\|_{H^1}<\varepsilon.
\end{align}
Since $Q_m^{(N)}$ is a finite sum and $P_\ell^{A_m}\to P_\ell^{A}$ on $H^1$ for each $|\ell|\le N$,
we have $Q_m^{(N)}u\to Q_A^{(N)}u$ in $H^1$, hence for $m$ sufficiently large,
\begin{align}
\|(I-Q_m^{(N)})u\|_{H^1}
\le \|(I-Q_A^{(N)})u\|_{H^1}+\|Q_m^{(N)}u-Q_A^{(N)}u\|_{H^1}
<2\varepsilon.
\end{align}
Decompose
\begin{align}
\big\langle U_m(t)u,\,v\big\rangle_{H^1}-\big\langle U_A(t)u,\,v\big\rangle_{H^1}
&=
\mathrm{(I)}_{m,N}+\mathrm{(II)}_{m,N}+\mathrm{(III)}_{N},
\end{align}
where
\begin{align}
\mathrm{(I)}_{m,N} &:= \big\langle (U_m(t)-U_m^{(N)}(t))u,\,v\big\rangle_{H^1},\\
\mathrm{(II)}_{m,N} &:= \big\langle (U_m^{(N)}(t)-U_A^{(N)}(t))u,\,v\big\rangle_{H^1},\\
\mathrm{(III)}_{N} &:= \big\langle (U_A^{(N)}(t)-U_A(t))u,\,v\big\rangle_{H^1}.
\end{align}
For each fixed cluster, all its phases $e^{it\lambda_k(A_m)}$ tend to $e^{it\mu_\ell(A)}$. The finitely many normalized eigenspinors have uniform $H^1$ bounds and bounded coefficient functionals. Thus replacing these phases by the limiting phase produces an error tending to zero in $H^1$. Together with $P_\ell^{A_m}\to P_\ell^A$, this gives $U_m^{(N)}(t)u\to U_A^{(N)}(t)u$ and hence $\mathrm{(II)}_{m,N}\to0$.
For the tails, note that
\begin{align}
U_m(t)-U_m^{(N)}(t)=U_m(t)\,(I-Q_m^{(N)}),
\qquad
U_A(t)-U_A^{(N)}(t)=U_A(t)\,(I-Q_A^{(N)}),
\end{align}
and therefore, by Cauchy--Schwarz and uniform boundedness on $H^1$,
\begin{align}
|\mathrm{(I)}_{m,N}|
&\le \|U_m(t)\|_{B(H^1)}\,\|(I-Q_m^{(N)})u\|_{H^1}\,\|v\|_{H^1}
\le 2\varepsilon\,\|U_m(t)\|_{B(H^1)}\,\|v\|_{H^1},\\
|\mathrm{(III)}_{N}|
&\le \|U_A(t)\|_{B(H^1)}\,\|(I-Q_A^{(N)})u\|_{H^1}\,\|v\|_{H^1}
\le \varepsilon\,\|U_A(t)\|_{B(H^1)}\,\|v\|_{H^1}.
\end{align}
Letting $m\to\infty$ first (with $N$ fixed) eliminates $\mathrm{(II)}_{m,N}$, and then letting $\varepsilon\downarrow 0$
yields the desired weak convergence.
\end{proof}

\medskip

\subsubsection*{Kernel distributional convergence.}
For smooth spinors $u,v$, using the integral representation of $U_m(t)$ and $U_A(t)$, we obtain
\begin{align}
\int_M \big\langle (U_m(t)u)(x),\,v(x)\big\rangle\,\dv_g(x)
&=
\int_M\!\!\int_M \big\langle K_m(t,x,y)\,u(y),\,v(x)\big\rangle\,\dv_g(y)\,\dv_g(x),\\
\int_M \big\langle (U_A(t)u)(x),\,v(x)\big\rangle\,\dv_g(x)
&=
\int_M\!\!\int_M \big\langle K_A(t,x,y)\,u(y),\,v(x)\big\rangle\,\dv_g(y)\,\dv_g(x),
\end{align}
and Proposition~\ref{prop:WOT-Awave} implies convergence of these pairings as $m\to\infty$.
Under the assumptions of Proposition~\ref{prop:WOT-Awave}, the convergence of associated wave operators
\begin{align}
e^{it(A_m^{-1}\D_g)}\rightarrow e^{it(A^{-1}\D_g)}\quad as\quad m\rightarrow\infty
\end{align}
as operators on the space of~$H^1$ spinors in the weak operator topology is verified.

Thus the stated spectral continuity and uniform propagator bound imply convergence of the wave kernels in these pairings. Individual boundedness of the propagators would not suffice to control the tails.

Furthermore, based on Theorem~\ref{thm:continuity-bdd}, and by following an argument entirely analogous to the one above, we obtain the corresponding convergence result for compact spin Riemannian manifolds with boundary, under the same smoothness, strong $L^2$ convergence, and uniform propagator bound on the fixed chiral $H^1$ domain. Namely, the associated wave operators satisfy
\begin{align}
e^{itA_m^{-1}(\D_g^{\CHI^+})}\rightarrow e^{it(A^{-1}\D_g^{\CHI^+})} \quad \text{as } m \to \infty,
\end{align}
as operators acting on the space of $H^1$ spinors in the weak operator topology.
\


\subsection{Beyond the spinorial setting: A general min--max characterization for weighted 1st order elliptic eigenvalue problems}
We emphasize that the mechanism behind complete continuity is not specific to the spinorial Dirac operator.
Rather, it is driven by a new min--max characterization combined with the Sobolev embedding theorem for bundle-valued sections and elliptic regularity theory. We outline here a general framework in which our arguments extend with only notational changes.

\subsubsection*{Elliptic operators on closed manifolds.}
Let $(M^n,g)$, $n\ge2$, be a closed smooth Riemannian manifold, and let $E,F\to M$ be real (complex) vector bundles equipped
with fiberwise inner products (Hermitian metrics). We denote by $C^\infty(M,E)$ the space of smooth sections and by $H^s(M,E)$ the
corresponding Sobolev spaces.

Let $P:C^\infty(M,E)\to C^\infty(M,F)$ be a first-order linear differential operator. For each $s\in\mathbb{R}$,
$P$ extends continuously to a bounded operator
\begin{align}
    P_s:H^{s+1}(M,E)\longrightarrow H^{s}(M,F).
\end{align}
The operator $P$ is called (uniformly) elliptic if its principal symbol $\sigma(P)(x,\xi):E_x\to F_x$ is
invertible for all $(x,\xi)\in T^*M\setminus\{0\}$.
We refer to~\cite{Bar1999zero,BarBandara2022boundary} for more examples of the general first order elliptic operators.

On a closed manifold, ellipticity yields the fundamental elliptic estimate: for any $s\in\mathbb{R}$ there exists
$C>0$ such that for all $u\in H^{s+1}(M,E)$,
\begin{equation}\label{eq:fundamental-elliptic-estimate}
\|u\|_{H^{s+1}(M,E)} \le C\Big(\|Pu\|_{H^{s}(M,F)}+\|u\|_{H^{s}(M,E)}\Big).
\end{equation}
In particular, $P_s$ is Fredholm between the corresponding Sobolev spaces, and elliptic regularity holds.
If, moreover, $P:C^\infty(M,E)\to C^\infty(M,E)$ is formally self-adjoint and elliptic of first order, then its
self-adjoint realization on $L^2(M,E)$ has compact resolvent, hence a real discrete spectrum with finite
multiplicities accumulating only at $\pm\infty$.

\subsubsection*{Sobolev--Morrey embeddings for bundle-valued sections.}
We also record standard Sobolev and Morrey embeddings for bundle-valued sections, which will be used implicitly
in the integrability regime $p>n$.

Let $\pi:E\to M$ be a smooth real or complex vector bundle of finite rank, equipped with a bundle metric
$\langle\cdot,\cdot\rangle_E$ and a compatible connection $\nabla$.
For $1\le p<\infty$, define the first-order Sobolev norm of a smooth section $u\in C^\infty(M,E)$ by
\[
\|u\|_{W^{1,p}(M,E)} := \|u\|_{L^p(M,E)} + \|\nabla u\|_{L^p(M,T^*M\otimes E)}.
\]
Let $W^{1,p}(M,E)$ be the completion of $C^\infty(M,E)$ with respect to this norm. Any two compatible connections
yield equivalent norms on $W^{1,p}(M,E)$.

\begin{lemma}\label{lem:sobolev-bundle}
Let $(M^n,g)$ be a closed Riemannian manifold and $E\to M$ a smooth vector bundle of finite rank.
Fix $1<p<\infty$.
\begin{enumerate}
\item[(i)] \textup{(Sobolev embedding)}
Assume $1<p<n$ and set $p^\ast:=\frac{np}{n-p}$. Then there exists $C=C(M,g,E,p)>0$ such that for all
$u\in W^{1,p}(M,E)$,
\[
\|u\|_{L^{p^\ast}(M,E)} \le C\,\|u\|_{W^{1,p}(M,E)}.
\]
In particular, $W^{1,p}(M,E)\hookrightarrow L^r(M,E)$ continuously for every $r\in[p,p^\ast]$.
\item[(ii)] \textup{(Morrey embedding)}
Assume $p>n$ and set $\alpha:=1-\frac{n}{p}\in(0,1)$. Then there exists $C=C(M,g,E,p)>0$ such that every
$u\in W^{1,p}(M,E)$ admits a H\"older continuous representative with
\[
\|u\|_{C^{0,\alpha}(M,E)} \le C\,\|u\|_{W^{1,p}(M,E)}.
\]
Equivalently, $W^{1,p}(M,E)\hookrightarrow C^{0,\alpha}(M,E)$ continuously.
\end{enumerate}
\end{lemma}

\noindent
For instance, see \cite{LawsonMichelsohn1989spin} for Sobolev spaces and
elliptic estimates on closed manifolds. The bundle-valued embeddings follow by combining the Euclidean Sobolev and
Morrey inequalities in local trivializations with a finite partition of unity.

\subsubsection*{Weighted eigenvalue problems and complete continuity.}
To study the continuity of weighted Dirac eigenvalues with respect to the weight, in the closed case, we prove in Proposition~\ref{prop:minmax} a min--max characterization of the weighted spectrum. Under the standing integrability regime $p>n$, this characterization implies the following \emph{complete continuity} property: for any sequence with $A_m\rightharpoonup A$ weakly in $L^p$, one has $\lambda_k(A_m)\to\lambda_k(A)$ for each fixed $k$. These results partially generalize the continuity theory for weighted Laplacian eigenvalues in Euclidean domains~\cite{WenYangZhang2018complete}.

Being variational in nature, the min--max principle extends with only minor changes to a broader class of self-adjoint elliptic operators. More precisely, let $M$ be a closed manifold and let $P:\Gamma(M,E)\to\Gamma(M,E)$ be a formally self-adjoint elliptic differential operator of first order on a real vector bundle $E$ (see, e.g., \cite[Section~2]{Bar1999zero}). Given $p>n$ and a weight $A\in L^p\bigl(M,\End(E)\bigr)$ that is symmetric and positive definite almost everywhere, one may consider the weighted eigenvalue problem
\begin{equation}\label{eq:weighted-general}
P\psi=\lambda\,A\psi.
\end{equation}
Under $p>n$, the same min--max argument yields the analogous \emph{complete continuity} conclusions for the
corresponding weighted spectrum: for each fixed $k$ one has $\lambda_k(A_m)\to\lambda_k(A)$ whenever
$A_m\rightharpoonup A$ weakly in $L^p$. Moreover, for each isolated spectral point (or an isolated finite spectral cluster) for~\eqref{eq:weighted-general}, the
associated finite-dimensional spectral subspace depends continuously on the weight; equivalently, the corresponding
weighted spectral projections converge in operator norm on $H^s$, $0<s<1$, with the same cluster convention as above. The positive and negative indices range over the spectral branches that are present.

A primary source of examples on closed manifolds is the class of first-order self-adjoint elliptic operators discussed in \cite[Section~2]{Bar1999zero}. Notably, Dirac-type operators associated with Dirac bundles in the sense of \cite[Definition~II.5.2]{LawsonMichelsohn1989spin} (also called Clifford modules in \cite[Definition~3.30]{BGV1992heat}) fit naturally into this framework. As pointed out in \cite[Proposition~3.35]{BGV1992heat}, if $M$ is even-dimensional and spin, general Clifford modules are given by tensor products of the spinor bundle with an auxiliary vector bundle. Hence, it provides a large class of self-adjoint first-order elliptic operators on closed manifolds, and the situation is not substantially different from the spin case.

For manifolds with boundary, the discussion proceeds in parallel. As a first generalization one can consider boundary value problems for elliptic operators of first order, as in~\cite{BarBandara2022boundary}. For the min--max argument to apply, one needs the boundary value problem to have a real spectrum consisting of eigenvalues; moreover, the boundary condition should be local. Under suitable local self-adjoint elliptic boundary conditions, an analogous min--max characterization, see Proposition~\ref{prop:minmax bdd}, and the corresponding continuity conclusions hold. Extensions to higher-order self-adjoint elliptic operators require the corresponding domains, Sobolev exponents, and boundary estimates; they are not discussed here.

\ 

\textbf{Conflict of Interest} The authors have no conflicts to disclose. 

\ 

\textbf{Data Availability Statement} Data sharing not applicable to this article as no datasets were generated or
analysed during the current study.

\printbibliography[
heading=bibintoc,
title={References}
]

\end{document}